\documentclass[reqno, 10 pt]{amsart}
\usepackage{amsmath, amsthm, amscd, amsfonts, amssymb, graphicx, color}
\usepackage{mathtools}

\usepackage[bookmarksnumbered, colorlinks, plainpages]{hyperref}
\usepackage[dvipsnames]{xcolor}
\newtheorem{theorem}{Theorem}[section]
\newtheorem{lemma}[theorem]{Lemma}

\newtheorem{corollary}[theorem]{Corollary}
\newtheorem{definition}{Definition}

\newtheorem{remark}[theorem]{Remark}
\numberwithin{equation}{section}
\newcommand{\M}{\textit{M}}
\newcommand{\MM}{\mathcal{M}(n,\lambda,D,i_0)}
\newcommand{\MMM}{\mathcal{M}}
\newcommand{\lk}{\lambda_k}

\newcommand{\G}{\Gamma}
\newcommand{\R}{\mathbb{R}}

\newcommand{\C}{\mathcal{C}_{\mathcal{M},k}}

\newcommand{\LL}{\lambda_{\MMM,k}}

\newcommand{\lbtheta}{(1-c_\MMM r)}
\newcommand{\CC}{C_{\MMM,k}}

\newcommand{\norm}[1]{\left\lVert#1\right\rVert}
\begin{document}

\setcounter{page}{1}
\title[Approximation of Laplacian]{Graph discretization of Laplacian on Riemannian manifolds with Bounds on Ricci curvature}

\author{ Soma Maity }
\address{Department of Mathematical Sciences, Indian Institute of Science Education and Research Mohali, \newline Sector 81, SAS Nagar, Punjab- 140306, India.}
\email{somamaity@iisermohali.ac.in}

\author{ Anusha Bhattacharya}
\address{Department of Mathematical Sciences, Indian Institute of Science Education and Research Mohali, \newline Sector 81, SAS Nagar, Punjab- 140306, India.}
\email{ph21059@iisermohali.ac.in}


\begin{abstract}

We study the approximation of eigenvalues for the Laplace-Beltrami operator on closed Riemannian manifolds in the class $\MMM$, characterized by bounded Ricci curvature, a lower bound on the injectivity radius, and an upper bound on the diameter. We use an 
$(\epsilon,\rho)$-approximation of the manifold by a weighted graph, as introduced by Burago et al. By adapting their methods, we prove that as the parameters $\epsilon, \rho$ and the ratio $\frac{\epsilon}{\rho}$ approach zero, the $k$-th eigenvalue of the graph Laplacian converges uniformly to the $k$-th eigenvalue of the manifold's Laplacian for each $k$. 

\end{abstract} \maketitle

\section{Introduction} 
\par The approximation of eigenvalues and eigenfunctions of the Laplacian on manifolds using graphs and simplices is a fundamental problem in spectral geometry and has extensive applications in theoretical and applied fields. Dodziuk, Patodi, and Petrunin introduced frameworks connecting Riemannian structures with triangulations and polyhedral approximations of manifolds, providing an early foundation for discrete Laplacian analysis in  \cite{MR488179}, \cite{MR407872}, \cite{MR1975337}. Later works by Fujiwara \cite{MR1257106} and Mantuano \cite{MR2130531} on spectral convergence on manifolds under general geometric conditions have provided much  of the conceptual basis for modern graph-based spectral approximation. A notable work by Belkin-Niyogi \cite{NIPS2006_5848ad95} established probabilistic convergence for random point samples. 
\par In the deterministic setting, a major advance made by Burago, Ivanov, and Kurylev \cite{article} established uniform spectral convergence of the graph Laplacian to the Laplace-Beltrami operator for closed Riemannian $n$-manifolds under uniform sectional curvature bounds, a lower injectivity radius bound, and an upper diameter bound. Aubry \cite{aubry2013approximation} obtained a similar result on approximating spectral data using isometrically immersed graphs in manifolds. Burago et al. \cite{MR3990939} and Lu \cite{MR4411102} also extended their framework to the $\rho$-Laplacian on metric measure spaces, including those with boundaries or certain singularities; further generalizations to vector bundles  appear in \cite{MR4437353}. Some recent related developments include \cite{MR4279237}, \cite{MR4384039}, \cite{e21010043}, \cite{MR4695859} and, \cite{MR4620352}. Applications of the spectral convergence of graph Laplacians also arise in manifold learning and numerical computation, see for instance [\cite{MR2504294},\cite{MR4130541},\cite{MR3636868},\cite{MR4273695}]. 

\par The methods of Burago et al. \cite{article} rely on estimates on the Jacobian of the exponential map derived from sectional curvature bounds. However, many fundamental results in Riemannian geometry, such as compactness theorems, eigenvalue bounds, and volume comparison, hold under the weaker assumption of Ricci curvature bounds. In this work, we establish a generalization of the spectral convergence theorem of \cite{article} to the broader class $\MM$ of closed $ n$-dimensional smooth Riemannian manifolds  satisfying the following conditions:
 $$|Ric|\leq \lambda, \quad diam \leq D, \quad {\rm and} \quad inj\geq i_0$$
 where $inj$ and $diam$ denote the injectivity radius and diameter of the manifold, respectively. Anderson proved that $\MM$ is precompact in the $C^{1,\alpha}$-topology via harmonic coordinates in \cite{anderson}, thereby extending Cheeger's finiteness theorem beyond its sectional curvature assumptions. The central contribution of our approach is to avoid dependencies on sectional curvature using estimates on the Jacobian of the exponential map derived from Anderson's precompactness theorem. Thus, adapting  the techniques in \cite{article} of approximating manifolds using weighted graphs dependent on $\epsilon, \rho>0$, we show that as $\epsilon,\rho $ and the ratio $\frac{\epsilon}{\rho}$ tend to 0, the $k$-th eigenvalue of the graph Laplacian converges uniformly to the $k$-th eigenvalue of the Laplace-Beltrami operator for every $k$.  

 \begin{definition}[\bf{$(\epsilon,\rho)$-approximation}] Given any $\rho>\epsilon>0$, a finite graph $\Gamma$ with the vertex set $ F=\{x_i\}_{i=1}^N$ and edges ${e_{ij}}$ is called an $(\epsilon,\rho)$-approximation of $(M,g)$ if the following conditions hold. 
\begin{itemize}
        \item $F\subset M$. There exists a partition of $M$ into measurable subsets $\{V_i\}_{i=1}^N$ such that $V_i\subset B_{\epsilon}(x_i)$ and $M=\bigcup B_{\epsilon}(x_i)$ where $B_{\epsilon}(x_{i})$ is the ball in $M$ centered at $x_{i}$ with radius $\epsilon$. 
        \item Two vertices $x_i$ and $x_j$ are connected by an edge $e_{ij}$ if the Riemannian distance between them is less than $\rho.$
        \item {\it Measure on $\Gamma$}:  Let $\mu_{i}=\text{vol}(V_{i}).$  $F$ can then be equipped with a discrete measure $\mu=\sum_{i=1}^{N}\mu_{i}\delta_{x_{i}}$ where $\delta_{x_i}$ denotes the Dirac measure at $x_i.$
         \item \textit{Weights on edges}: Let $\nu_{n}$ be the volume of the unit ball in the Euclidean $n$-space. To an edge $e_{ij}$, assign the weight 
    \begin{equation}
        w_{ij}=\frac{2(n+2)}{\nu_{n}\rho^{n+2}}\mu_{i}\mu_{j}.
    \end{equation}
        \end{itemize}
        \end{definition}
For any $\M \in\MM$, there exists a finite set of points $F$ that is $\epsilon$-dense in $\M$. We can then use a partition obtained by the Voronoi decomposition to define an $(\epsilon,\rho)$-approximation of $(M,g)$. The weighted graph is denoted by $\G(F,\rho,\mu).$ A weighted graph Laplacian  $\Delta_{\Gamma}$ on $L^{2}(\G)$ is defined as follows:
    \begin{align}
        (\Delta_{\Gamma}u)(x_{i})&=\frac{1}{\mu_{i}}\sum_{x_{i}\sim x_{j}}w_{ij}(u(x_{i})-u(x_{j}))\nonumber\\ 
        &= \frac{2(n+2)}{\nu_{n}\rho^{n+2}} \sum_{x_{i}\sim x_{j}}\mu_{j}(u(x_{i})-u(x_{j})).
    \end{align}
 where $x_i\sim x_j$ is $x_i$ and $x_j$ are connected by an edge.
 \par The motivation behind the choice of the normalization constant is given in Section 2.3 of \cite{article}. We note that $-\Delta_{\Gamma}$ is a non-negative self-adjoint operator on $L^2(\G)$ with eigenvalues $0\leq \lambda_1(\Gamma)\leq \lambda_2(\Gamma) ....\leq \lambda_N(\Gamma)$. The eigenvalues of $-\Delta_M$ are denoted by $0\leq \lambda_1(M)\leq \lambda_2(M)\leq \lambda_3(M)\leq ...$. The main goal of this paper is to prove the following theorem. 
\begin{theorem} \label{main 1}
Consider a manifold $M\in \MM$ and an $(\epsilon,\rho)$-approximation $\G$ of $M$ such that $\rho<i_0$. Let $\lambda_k(M)$, $\lambda_k(\G)$ denote the $k$-th eigenvalue of $-\Delta_M$ and $-\Delta_\Gamma$ respectively. Then  there exists a constant $C_{\MMM}>0$ such that for any $\rho, \frac{\epsilon}{\rho}<\frac{1}{C_{\MMM}}$,
\begin{equation} \label{indivisual bound}
     |\lk(\G)-\lk(M)|\leq C_{\MMM}\left(\frac{\epsilon}{\rho}+\rho \right)\lambda_k(M)+\rho C_{\MMM}\lambda^{\frac{3}{2}}_k(M), \quad \forall \ k\leq |\G|.\nonumber
   \end{equation}
\end{theorem}
In \cite{cheng}, Cheng proved that if the Ricci curvature of $M$ is bounded below by $-\lambda$ and the diameter is bounded above by $D$, then the $k$-th eigenvalue of $-\Delta_M$  is bounded above by a constant depending only on $k$, $n$, $\lambda$ and $D$. Hence, for any $M\in \MM$, there exists a uniform positive constant $\LL$ such that $\lambda_k\leq \LL.$ Consequently, we have the following corollary.
\begin{corollary}\label{cor1} Consider a manifold $M\in \MM$ and an $(\epsilon,\rho)$-approximation $\G$ of $M$ such that $\rho<i_0$. Let $\lambda_k(M)$, $\lambda_k(\G)$ denote the $k$-th eigenvalue of $-\Delta_M$ and $-\Delta_\Gamma$ respectively. Then for each $k$, there exists a constant $C_{\MMM,k}>0$ depending on $\MMM$ and $k$  such that for any $\rho, \frac{\epsilon}{\rho}<\frac{1}{C_{\MMM,k}}$,
\begin{equation} 
     |\lk(\G)-\lk(M)|\leq C_{\MMM,k}\left(\frac{\epsilon}{\rho}+\rho \right).\nonumber
   \end{equation}
Consequently, for every fixed $k$, $\lambda_k(\Gamma)$ converges uniformly to $\lambda_k(M)$ for all  $M\in \MM$ as  $\rho\to 0$ and $\frac{\epsilon}{\rho}\to 0.$    
 \end{corollary}

We now state the approximation of eigenfunctions of the Laplace-Beltrami operator by using the discretization scheme introduced above. We work with the discretization map  $P$ (Definition \ref{discretization defn}) that connects functions on the manifold with functions on the discrete space and the \emph{interpolation map} $I$ (Definition \ref{interpolation defn}) which reconstructs a function on the manifold from discrete data.  

\begin{theorem}
      Let $\lambda=\lambda_j(M)$ be an eigenvalue of $\Delta_M$ with multiplicity $m$, such that
$$ \lambda_{k-1} < \lambda_k = \lambda = \lambda_{k+m-1} < \lambda_{k+m}. $$
 Let $\delta_\lambda = \min\{1, \lambda_k - \lambda_{k-1}, \lambda_{k+m} - \lambda_{k+m-1}\}$.
Let $u_k, \dots, u_{k+m-1}$ be orthonormal eigenvectors  corresponding to  $\lambda_k(\Gamma), \dots, \lambda_{k+m-1}(\Gamma)$.

Then there exist orthonormal eigenfunctions $g_k, \dots, g_{k+m-1}$ of $-\Delta_M$ corresponding to $\lambda$ such that for all $j = k, \dots, k+m-1$ and $ \rho<\delta_\lambda C_{\mathcal{M}, k}^{-1}$,
\begin{equation}
\norm{u_j - P g_j}^2 \leq C_{\mathcal{M}, k} \delta_\lambda^{-2} \left(\frac{\epsilon}{\rho}+\rho \right)
\end{equation}
and
\begin{equation}
\norm{g_j - I u_j}^2 \leq C_{\mathcal{M}, k} \delta_\lambda^{-2}\left(\frac{\epsilon}{\rho}+\rho \right).
\end{equation}
\end{theorem}

\subsection*{Idea of the proof:} By Rauch's comparison theorem, if the absolute value of the sectional curvature is bounded by $K$, the Jacobian $J_x(v)$ of the exponential map at $x$ in the direction of $v$ satisfies
$$\frac{1}{1+CnK|v|^2}\leq J_x(v)\leq 1+CnK|v|^2 , \quad \forall \ |v|< i_0.$$ 
This inequality is a crucial component of the eigenvalue and eigenfunction approximation methods in  \cite{article}.  However, this type of bound can not be easily obtained for bounded Ricci curvature. Although the authors in \cite{article} noted the possibility of extending their methods to manifolds with Ricci curvature bounds, the generalization is not straightforward. We use Anderson's precompactness theorem in \cite{anderson} and derive appropriate estimates for the Jacobian in Lemma \ref{volformbound} that allow us to adopt the methods in \cite{article} in the case of bounded Ricci curvature. 
\par We prove Theorem \ref{main 1} using the Min–Max principle. In Section 3, a discretization map from $L^2(M)$ to $L^2(\G)$ is used to relate the discrete Dirichlet energy to the energy functional on $M$. Then by controlling the Rayleigh quotient on the image under the discretization map of the span of the first $k$ eigenfunctions of $-\Delta_M$, these estimates yield an upper bound for $(\lk(\G)-\lk(M))$.  
\par The lower bound for $(\lk(\G)-\lk(M))$ is obtained in Section 4. Using an interpolation map, we regularize elements of $L^2(\G)$ into $C^{0,1}(M)$. This map is an approximate inverse of the discretization map. We then establish relations between the derivative of the interpolation map and the discrete Dirichlet energy. Considering the Rayleigh quotient restricted to the image of the interpolation map of the span of the first $k$ eigenfunctions of $-\Delta_\G$, and using our interpolation estimates, we obtain the desired lower bound.
\section{Some inequalities and Average dispersion}

Given a Riemannian manifold $(M,g)$ and a point $x\in M$, the exponential map at $x$ is denoted by $\exp_x$. Let $u$ be a unit vector in $T_xM.$ For $0<t< i_0$, using the normal polar coordinates $(u,t)$ in $T_xM$, we can write the volume element $\text{{vol}}_g$ as
\begin{equation} \label{jacpull}
    \exp_x^*\operatorname{vol}_g=J(u,t)t^{n-1}dtdu. \nonumber
\end{equation}
where $J(u,t)$ is called the Jacobian determinant at $(u,t)$. In the sequel, we require suitable bounds $\lvert J(u,t) \rvert$ under the assumed geometric conditions. We first establish an estimate for the Laplacian of the distance function under an upper Ricci curvature bound.

\begin{lemma}[Laplacian comparison of the distance function]
\label{lem:laplace comparison}
Let $(M,g)$ be a Riemannian n-manifold satisfying  $\operatorname{Ric} \le \lambda $, and $\operatorname{inj}(M)\ge i_0>0$. Fix $p\in M$ and let $r(x)=d(p,x)$ be the distance function from $p$.
Then at $p\in B_{i_0/2}(x)\setminus\{x\}$,
\[
\Delta r (p)\ge \frac{n-1}{r(p)} - \lambda\, r(p).
\]
\end{lemma}

\begin{proof}
Let $\gamma:[0,r(p)]\to M$ be the unique unit-speed minimizing geodesic from $x$ to $p$. One can write in polar coordinates
\[g=dr^2+g_r.\]

Let $H_r$ denote the Hessian of the distance function $r$.   Then the fundamental equation (see \cite{petersen} Chapter 2, section 5.2) can be written as  \[\partial_rg=\partial_rg_r=2H_r\] and 

\[
\partial_rH_r -H_r^2 =- \operatorname R(.,\partial_r,\partial_r,.).
\]
By a taking trace and using the curvature assumption $\operatorname{Ric}\le \lambda$,
we have 
\begin{equation}
  \label{eq:fundamental}
\partial_r\Delta r \geq
\operatorname{tr}(H_r^2)-\lambda\geq \frac{(\Delta r)^2}{n-1}
 -\lambda.  
\end{equation}
The last inequality follows from the Cauchy-Schwarz inequality. Therefore,
\begin{align}
   \partial_r \left(\Delta_r-\frac{n-1}{r}\right)&=\partial_r\Delta r +\frac{n-1}{r^2} -2\frac{\Delta r}{r} \nonumber \\ &\geq \frac{1}{n-1}(\Delta r)^2 -\lambda + \frac{n-1}{r^2} -2\frac{\Delta r}{r} \label{fundamental equation bound} \\
   &=\frac{1}{n-1}\left(\Delta r - \frac{n-1}{r}\right)^2 -\lambda\nonumber\\ 
   &\geq -\lambda.
\end{align}
 Let $f(r)=\Delta r-\frac{n-1}{r}.$ Then,
\begin{align}
    f(r)-f(\epsilon)&=\int_\epsilon^rf'(s)ds \nonumber\\
    &\geq -\int_{\epsilon}^r\lambda ds\nonumber\\
    &=-\lambda(r-\epsilon).
\end{align}
From \cite{petersen} Chapter 5, section 7
$$\lim_{r\to 0}\left( H_r-\frac{1}{r}g_r\right)=0.$$
Taking the trace we have,
\[\lim_{\epsilon\rightarrow 0}\left(\Delta r(\gamma(\epsilon))-\frac{n-1}{\epsilon}\right)=0.\] Hence,
\[f(r)\geq -\lambda r.\]
Which implies, 
\[\Delta r(p)\geq \frac{n-1}{r(p)}-\lambda r(p).\]

\end{proof}

\begin{lemma} \label{volformbound}  
 There exists a constant $c_\MMM>0$ such that for any $M\in \MM$ and $x\in M$, 
 $$|1-J(u,t)|\leq c_\MMM t ,\quad \forall \ u\in S^{n-1}  , \quad \forall \ t  \leq  \frac{i_0}{2}.$$
 \end{lemma}
    \begin{proof}
   Consider a polar coordinate decomposition on $B_r(x)$, the geodesic ball of radius $r\le \frac{i_0}{2}$ centered at $x$ in $(M,g)\in\MMM.$ Then, from \cite{petersen}, Chapter 9, section 1 we have
   \[d\operatorname{vol}=\Tilde{J}(u,t)du\wedge d\operatorname{vol}_{n-1}\]
   and \begin{equation}
       \label{eq:jacobian relation with laplacian}
       \partial_r\Tilde{J}=\Tilde{J}\Delta r.
   \end{equation}
   where $\Tilde{J}(u,t)=J(u,t)t^{n-1}.$ By differentiating (\ref{eq:jacobian relation with laplacian}), we obtain
   
    \begin{equation}
    \label{eq:laplace r and jacobian relation}
        \Delta r=\frac{n-1}{r}+\frac{J'(u,r)}{J(u,r)}.
    \end{equation}
From Lemma \ref{lem:laplace comparison} and the Laplace comparison theorem (see \cite{petersen}, Chapter 9, section 1.2) for $\operatorname{Ric}\geq -\lambda$, we have
\[\frac{n-1}{t}-\lambda t\leq \Delta r(\gamma(t))\leq (n-1)\sqrt{\lambda}\coth(\sqrt{\lambda}t).\]
  Using (\ref{eq:laplace r and jacobian relation}),
  \[-\lambda tJ(u,t)\leq J'(u,t)\leq (n-1)\sqrt{\lambda}\coth(\sqrt{\lambda}t)J(u,t)-\frac{n-1}{t}J(u,t).\]
  From the Taylor series expansion of the $\coth$ function, we have $\lim_{t\rightarrow 0}\left(\sqrt{\lambda}\coth(\sqrt{\lambda} t)-\frac{1}{t}\right)=0$ and $J(u,t)\leq \left(\frac{\sinh (\sqrt{\lambda}t}{\sqrt{\lambda}t}\right)^{n-1}.$ Using $|J(u,t)|\leq C(n,\lambda,i_0)$, we obtain a uniform upper bound $c_{\MM}$ on the derivative of $J(u,t)$ on $B_x(r)$ for $r\le \frac{i_0}{2}.$ \begin{equation}
      \label{eq:derivative bound}
      |J'(u,t)|\leq c_{\MMM}, \quad \forall p=\exp_x(tu)\in B_{\frac{i_0}{2}}(x),\quad \forall M\in\MMM .
  \end{equation}
 
   From the proof of the Bishop-Gromov volume comparison theorem, we know that $J(u,0)=1$ for any $u\in S^{n-1}.$  Using mean value theorem and (\ref{eq:derivative bound})  we have 
       \begin{equation}
          \nonumber
           |J(u,t)-1|=|J(u,t)-J(u,0)|\leq c_\MMM  t. 
       \end{equation}
    
\end{proof}
\begin{definition} Let $\M\in \MM$ and $B_{r}(x)$  be a ball in $\M$ centered at $x$ of radius $r$. Let $f:\M\to \mathbb{R}$ be a smooth function. The average dispersion $E_{r}(f)$ is defined as 
\begin{equation*}
  E_{r}(f)=\int_{\M}\int_{B_{r}(x)}|f(y)-f(x)|^{2}dy dx . 
\end{equation*}
\end{definition}
We recall that $H^1(M)$ is the space of all functions in $L^2(M)$ such that their first-order partial derivatives are also in $L^2(M)$. The following lemma gives an upper bound of $E_r(f)$.
\begin{lemma}\label{theo1} For $M\in \MM$, let $0<r<\frac{i_0}{2}, f\in H^{1}(\M) $. Then
\begin{equation}
    E_{r}(f)\leq \frac{\nu_{n}r^{n+2}}{n+2}(1+c_\MMM r) \norm{df}^2. \nonumber
\end{equation}

\end{lemma}
\begin{proof} Since smooth functions are dense in $H^1(M)$, it is enough to prove the above inequality for smooth functions. 
For any $r<\frac{i_{0}}{2},$ the exponential map at \textit{x} restricted to $B_{r}(x)$ is a diffeomorphism onto its image. Using polar co-ordinates $(u,t)$ on $T_xM$,  
\begin{align} \label{coordinatechange}
    \int_{B_{r}(x)}\lvert f(y)-f(x)\rvert^{2}dy=\int_{u\in S^{n-1}}\int_{t=0}^r\lvert f(exp_{x}(tu))-f(exp_{x}(0))\rvert^{2}J(u,t)t^{n-1}dtdu. 
\end{align}
 Substituting the upper bound of $J(u,t)$  from Lemma \ref{volformbound} in equation (\ref{coordinatechange}), 
 \begin{equation*}
  \int_{B_{r}(x)}|f(y)-f(x)|^{2}dy\leq (1+c_\MMM r)\int_{u\in S^{n-1}}\int_{t=0}^r\lvert f(exp_{x}(tu))-f(exp_{x}(0))\rvert^{2}t^{n-1}dtdu. 
 \end{equation*}
Now,
 \[\int_{u\in S^{n-1}}\int_{t=0}^r\lvert f(exp_{x}(tu))-f(exp_{x}(0))\rvert^{2}t^{n-1}dtdu=\int_{B_{r}(0)\subset T_{x}\M}|f(exp_{x}(v))-f(x)|^{2}dv.\]
By using the techniques for establishing an upper bound of the same integral in Lemma 3.3 (\cite{article}), we have
 \begin{equation}
   \int_{B_{r}(0)\subset T_{x}\M}|f(exp_{x}(v))-f(x)|^{2}dv \leq \frac{\nu_{n}r^{n+2}}{n+2}\norm{df}^{2}_{L^{2}}. \nonumber
 \end{equation}
 Combining this with the upper bound of the Jacobian determinant, the required inequality is obtained.  
 \end{proof}
In \cite{MR448253berger}, Berger showed that there exists a constant $c_n>0$ depending only on $n$ such that
\begin{equation} \label{volume bd}
    \text{vol}(B_r(x))\geq c_nr^n, \quad \forall r<i_0.
\end{equation}
Its values have been computed in  \cite{croke}. We establish a Poincaré-type inequality using this estimate. 
\begin{lemma}\label{poincare ineq corollary} For $0<\epsilon<r< \frac{i_0}{2}$, let  \textit{V} be a measurable subset of M with $diam \text{ }V\leq \epsilon$ such that \normalfont{\text{vol}}$(V)=\mu>0$. For all $x\in M$, let $\normalfont{\text{vol}(B_r(x))}\geq c_nr^n. $ Let $f\in L^{2}(M)$ and $a=\frac{1}{\mu}\int_{V}f(x)dx$ be the average value of $f$ on $V.$  Then 
    $$    \int_{V}|f(x)-a|^{2}dx\leq \frac{1}{c_n(r-\epsilon)^{n}}E_{r}(f,V).$$
\end{lemma}
\begin{proof} The above inequality is obtained in Lemma 3.4 in \cite{article} where an upper bound on the sectional curvature is used to obtain a lower bound on  $\text{\normalfont{vol}}(B_r(x))$. If $\text{\normalfont{vol}}(B_r(x))\geq c_nr^n$, then the proof follows using the same techniques. 
\end{proof}
\section{An upper bound for $\lambda_{k}(\Gamma)$}
We find an upper bound of the eigenvalues of $-\Delta_\G$ in terms of the eigenvalues of $-\Delta_M$ for $M\in\MM$ using the inequalities proved in the previous section. Consider $M\in \MM$ and $\epsilon>0$ such that $0<\epsilon<\rho<\frac{i_0}{2}$. Let $\G$ be an $(\epsilon,\rho)$-approximation of $(M,\normalfont{\text{vol}})$ using a decomposition $\{V_{i}\}_{1}^{N}$. We define the discretization map that assigns to each vertex of the graph the average value of the function on $V_i$, the measurable set that contains the vertex.
\begin{definition}[Discretization map] \label{discretization defn}
 Let $\mu_{i}=\normalfont{\text{vol}}(V_{i})$. Then $P: L^{2}(M)\to L^{2}(\G)$  be a map defined as 
 \begin{equation*}
     Pf(x_{i})=\frac{1}{\mu_{i}}\int_{V_{i}}f d\mu .\end{equation*}
\end{definition}
The map
     $P^*: L^{2}(\G)\to L^{2}(M)$ is defined as 
     \begin{equation*}
       P^*(u)=\sum_{i=1}^{N}u(x_{i})\chi_{V_{i}}  .
     \end{equation*}
We see some immediate properties of the discretization map in the following lemma.
\begin{lemma}  Let $M\in\MM, f\in L^2(M)$ and $u\in L^2(\G).$ Then
\begin{enumerate}
    \item $P: L^{2}(M)\to L^{2}(\G)$ is a bounded linear operator :
    $\norm{Pf}\leq N\norm{f}_{2}.$
    \item $P^*$ is norm-preserving, i.e. $\norm{P^*(u)}_{L^{2}(M)}=\norm{u}_{L^{2}(\G)}$.
    \item $P^*$ is the adjoint of the discretization map $P$, i.e. $\langle f, P^*(u)\rangle_{L^{2}(M)}=\langle P(f), u\rangle_{L^{2}(\G)}$.
\end{enumerate}
\end{lemma}
\begin{lemma} If $0<\epsilon<\frac{\rho}{n}<\frac{i_0}{2n}$ and $\normalfont{\text{vol}}(B_r(x))\geq  c_nr^n$ for $\epsilon<r<\rho$ then \label{P*Pf-f}
\begin{equation*}
 \norm{f-P^{*}Pf}^{2}_{L^{2}} \leq \frac{4n\nu_{n}}{c_n }(1+c_\MMM r)\epsilon^2\norm{df}^2_{L^{2}}.
\end{equation*}
\end{lemma}
\begin{proof}
Let $V$ be a measurable subset of $M$ and $\chi_V$ denote the characteristic function on $V.$ According to definition,
\begin{align}
    \norm{f-P^{*}Pf}^{2}_{L^{2}}&=\int_{M}|f(x)-P^{*}Pf(x)|^{2}dx \nonumber \\ &=\int_{M}|f(x)-\sum_{i=1}^{N}Pf(x_{i})\chi_{V_{i}}(x)|^{2}dx \nonumber \\ &=\sum_{i=1}^{N}\int_{V_{i}}|f(x)-Pf(x_{i})|^{2}dx.
\nonumber\end{align} 
From Lemma \ref{poincare ineq corollary}, for all $r$ such that $0<\epsilon<r<\rho$ and for all $1\leq i\leq N$,
\begin{equation}
    \int_{V_{i}}|f(x)-Pf(x_{i})|^{2}dx \leq \frac{1}{c_n(r-\epsilon)^{n}}E_{r}(f,V_{i}). \nonumber
\end{equation}    
From Lemma \ref{theo1},
\begin{equation}
    \norm{f-P^{*}Pf}^{2}_{L^{2}}\leq \frac{1}{c_n(r-\epsilon)^{n}}E_{r}(f) \leq \frac{\nu_n}{c_n(n+2)}\left(\frac{r}{r-\epsilon}\right)^n r^2 (1+c_\MMM r)\norm{df}^2_{L^{2}}.  \nonumber
\end{equation}
Putting $r=n\epsilon$, we have $\left(\frac{r}{r-\epsilon}\right)^{n}=\left(\frac{n}{n-1}\right)^{n}<4.$ 
Hence,  \begin{equation} \label{bound of f and PP*}
    \norm{f-P^{*}Pf}^{2}_{L^{2}} \leq\frac{4n\nu_{n}}{c_n }(1+c_\MMM r)\epsilon^2 \norm{df}^2_{L^{2}}.%\leq 4n \Ju( r)\epsilon^2 \norm{df}^2_{L^{2}}.
\end{equation}
\end{proof}
As $P^{*}$ is norm-preserving, the next corollary follows immediately using the triangle inequality.
\begin{corollary} \label{cor2} For any $f\in H^1(M),$
\begin{equation}
|\norm{Pf}-\norm{f}_{L^{2}}|^2\leq \frac{4n\nu_{n}}{c_n } (1+c_\MMM r)\epsilon^2 \norm{df}^2_{L^{2}}.      \nonumber 
\end{equation}
\end{corollary}
\begin{proof}
    \begin{equation}
     |\norm{Pf}-\norm{f}_{L^{2}}|^2 \leq |\norm{P^{*}Pf}-\norm{f}|^2 \leq \frac{4n\nu_{n}}{c_n } (1+c_\MMM r)\epsilon^2 \norm{df}^2_{L^{2}}.  \nonumber   
    \end{equation}
\end{proof}    
Next, we define the discrete Dirichlet energy and study its relation to the Dirichlet energy functional defined on $H^1(M).$ 
 \begin{definition}[Discrete Dirichlet energy] Let $\G=\Gamma(F,\rho,\mu)$ be an $(\epsilon,\rho)$- approximation of $M\in\MM.$ For $u\in L^2(\G)$, the Dirichlet energy on $\G$ is given by
 \begin{equation}
 \norm{\delta(u)}^{2}=\frac{n+2}{\nu_{n}\rho^{n+2}}\sum_{i}\sum_{j:x_{i}\sim x_{j}}\mu_{i}\mu_{j}|u(x_{j})-u(x_{i})|^{2}.    
 \end{equation}
 \end{definition}
 \begin{lemma} \label{delp} Let  $\G$ be a $(\epsilon,\rho)$-approximation of $(M,\normalfont{\text{vol}})$. For $ f\in H^1(M)$, 
$$
     \norm{\delta(Pf)}^2\leq (1+c_\MMM (\rho+2\epsilon))\left(1+\frac{2\epsilon}{\rho}\right)^{n+2}\norm{df}^{2}_{L^{2}}. $$
    
 \end{lemma}  
 \begin{proof}
    From  Lemma 4.3(2) in \cite{article}, we derive
      \begin{align}
         \norm{\delta(Pf)}^{2}&\leq \frac{n+2}{\nu_{n}\rho^{n+2}}\sum_{i}\sum_{j:x_{j}\sim x_{i}}\int_{V_{i}}\int_{V_{j}}|f(y)-f(x)|^{2}dydx \nonumber \\ &= \frac{n+2}{\nu_{n}\rho^{n+2}} \int_{M}\int_{U(x)}|f(y)-f(x)|^{2}dydx. \label{ux}
     \end{align}
where $U(x)=\bigcup_{j:x_{j}\sim x_{i}}V_{j}$ in  (\ref{ux}). We also observe that $U(x)\subseteq B_{\rho+2\epsilon}(x).$ \\ Hence, \begin{equation} \label{rhoeps}
     \norm{\delta(Pf)}^{2}\leq \frac{n+2}{\nu_{n}\rho^{n+2}}E_{\rho+2\epsilon}(f).    
\end{equation}
Using Lemma \ref{theo1} and from  (\ref{rhoeps}), we have \begin{align}
      \norm{\delta(Pf)}^{2}&\leq (1+c_\MMM (\rho+2\epsilon))\frac{(\rho+2\epsilon)^{n+2}}{\rho^{n+2}}\norm{df}^{2}_{L^{2}} \nonumber \\ &=  (1+c_\MMM(\rho+2\epsilon))\left(1+\frac{2\epsilon}{\rho}\right)^{n+2}\norm{df}^{2}_{L^{2}}. \nonumber
     \end{align} \end{proof}   
Now we can calculate an upper bound of the eigenvalue of the graph Laplacian operator using the lemmas proved in this section.
 \begin{theorem} \label{upperbd}
Let $0<\epsilon<\rho\leq\frac{i_0}{2} \text{, } M\in\MM $  and $\G$ be a $(\epsilon,\rho)$-approximation of $(M,\normalfont{\text{vol}})$.  Let $\lk(M)$ and $\lambda_{k}(\Gamma)$ denote the k-th eigenvalues of $-\Delta_{M}$  and $-\Delta_{\G}$ respectively. There exist positive constants $C_n$ and $C_{\mathcal{M}}$ depending only on $n$ and $\MMM$ respectively  such that for any $\rho$, $\frac{\epsilon}{\rho}< \frac{1}{C_n}$ ,
$$ \lambda_{k}(\Gamma)\leq \left(1+C_n\frac{\epsilon}{\rho}+C_{\mathcal{M}}\rho \right)\lambda_k(M)+\rho C_{\MMM}\lambda^{\frac{3}{2}}_k(M).$$  
\end{theorem}
\begin{proof} We use the min-max principle to prove the theorem. Hence it is sufficient to find a linear subspace $L\subset L^{2}(\G) $ with dim$ L=k$ such that the Rayleigh quotient restricted to $L$ given by $ sup_{u\in L\setminus \{0\}} \frac{\norm{\delta u}^{2}}{\norm{u}^{2}}$ is less than or equal to the right-hand side of the above inequality. 
\par Let $W\subset H^{1}(M)$ be the linear span of orthonormal eigenfunctions of $-\Delta_{M}$ corresponding to eigenvalues $\lambda_{1}(M)\leq \lambda_2(M)\leq...\leq \lambda_{k}(M).$ For any $f \in W,$ 
    $$\norm{df}^{2}_{L^{2}}=\langle -\Delta_Mf,f\rangle \leq \lambda_{k}(M)\norm{f}^{2}_{L^{2}}.$$
    From Corollary \ref{cor2} and considering $\rho<\frac{1}{c_\MMM}$, \begin{equation}
        \norm{Pf}\geq \norm{f}_{L^{2}}-2\epsilon \sqrt{\frac{n\nu_{n}}{c_n }(1+\rho c_\MMM)
        }\norm{df}_{L^{2}} \geq \left(1-3\epsilon \sqrt{\frac{n\nu_{n}}{c_n }\lk(M)}\right)\norm{f}_{L^{2}}. \label{uppbd}
    \end{equation}
Let $\lambda_{\mathcal{M},k}$ be an upperbound of $\lambda_k(M)$ for any $M\in \MMM.$ Then for $\epsilon<\frac{\sqrt{c_n}}{3\sqrt{n\nu_n \LL}}$, $P|_{W}$ is injective.  Hence, the dimension of the subspace $L=P(W)$ is $k.$ We pick $u\in L\setminus\{0\}$ and let $f \in W$ be such that $u=Pf.$ Then from equation (\ref{uppbd}), 
  \begin{equation} \label{dim}
      \norm{u}^{2}\geq \left(1-3\epsilon \sqrt{\frac{n\nu_{n}}{c_n }\lk(M)}\right)\norm{f}^{2}_{L^{2}}. 
  \end{equation}
 By Lemma \ref{delp} for $\epsilon<\frac{\rho}{2}$,
   \begin{align}
      \norm{\delta(u)}^{2}\leq  (1+2\rho c_\MMM)\left(1+\frac{2\epsilon}{\rho}\right)^{n+2} \lk(M)\norm{f}^{2}_{L^{2}}. \label{num1}
  \end{align}
  From inequalities (\ref{dim}) and (\ref{num1}),  we have 
  \begin{align} \label{upperbd}
   \lambda_{k}(\Gamma)\leq  \frac{\norm{\delta u}^{2}}{\norm{u}^{2}}&\leq  \frac{ (1+2\rho c_\MMM)\left(1+\frac{2\epsilon}{\rho}\right)^{n+2} \lk(M)}{\left(1-3\epsilon \sqrt{\frac{n\nu_{n}}{c_n }\lk(M)}\right)}. \nonumber
  \end{align}
Using the Taylor series expansion of terms $\left(1+\frac{2\epsilon}{\rho}\right)^{n+2}$ and $\left(1-3\epsilon \sqrt{\frac{n\nu_{n}}{c_n }\lk(M)}\right)^{-1}$ we obtain a constant $C_{\mathcal{M}}$ depending only on $\MMM$ and a constant $C_n$ depending on $n$ such that for any $\epsilon, \frac{\epsilon}{\rho}< \frac{1}{C_n}$ ,
$$ \lambda_{k}(\Gamma)\leq \left(1+C_n\frac{\epsilon}{\rho}+C_{\mathcal{M}}\rho \right)\lambda_k(M)+\rho C_{\MMM}\lambda^{\frac{3}{2}}_k(M).$$ 
 \end{proof}
\begin{remark}
   We observe that it is enough to assume a lower bound on Ricci curvature to obtain an upper bound on $J(u,t)$, which will give us a lower bound on $\lk(M)$ using $\lk(\G)$. One can apply Bishop- Gromov volume comparison theorem  and conclude that for Ricci curvature bounded below by $\lambda$ and $t<i_0$, $J(u,t)\leq J_\lambda(t)$ where 
\begin{equation}
        J_\lambda(t)=\left(1+\sum_{j\in\mathbb{N}}\frac{|\lambda|^jt^{2j}}{(2j+1)!}\right)^{n-1}.
    \end{equation}
   We can then follow the calculations of  Theorem \ref{upperbd} and obtain a lower bound of $\lk(M).$
\end{remark}
\section{Smoothing Operator}
 In this section, we study the smoothing operator defined in \cite{article} and its properties, which will help us to obtain a lower bound for the eigenvalues of $-\Delta_\G$ using eigenvalues of $-\Delta_M$.
\begin{definition} Fixing a positive $r<\rho<\frac{i_0}{2}$, the following functions are defined:
\begin{enumerate}
    \item \textbf{Smoothing function.} Let $\psi: \mathbb{R}_{+}\to \mathbb{R}_{+}$ be defined as 
    \[   
\psi(t)=
     \begin{cases}
        \frac{n+2}{2\nu_n}(1-t^2)&\quad\text{if } 0\leq t\leq 1\\        0&\quad\text{if }t\geq 1.\\

     \end{cases}
\]
\item \textbf{Kernel.} Let $k_r: M\times M \to \mathbb{R}_+$ be defined as 
\begin{equation*}
    k_r(x,y)=r^{-n}\psi(r^{-1}d(x,y)).
\end{equation*}
where $d(x,y)$ denotes the Riemannian distance between $x$ and $y$ in $M.$

\item \textbf{Associated integral operator.} Let $\Lambda_r^0:L^2(M)\to C^{0,1}(M)$ be given by 
\begin{equation*}
    \Lambda_r^0f(x)=\int_M f(y)k_r(x,y)dy.
\end{equation*}
\item Let $\theta\in C^{0,1}(M)$ be defined as \begin{equation*}
    \theta=\Lambda_r^0(\mathbf{1}_M).
\end{equation*}
\end{enumerate} 
\end{definition}
By using the polar coordinate in $\mathbb{R}^n$, the following lemma can be easily established.
\begin{lemma} \label{integral 1}
    Let $p\in\mathbb{R}^n$ and $ x,y \in M $. Let $\psi, k_r$ be defined as above. Then, 
    \begin{enumerate}
        \item $\int_{\R^n}\psi(|p|)dt=1$, and
        \item $|k_r(x,y)|\leq\frac{n+2}{\nu_n r^n}.$
        \item ${\rm grad} k_r(.,y)(x)=\frac{n+2}{\nu_nr^{n+2}}\exp^{-1}(y)$
    \end{enumerate}
\end{lemma}
We derive the following approximations for $\theta$ using the bounds of the Jacobian determinant from Lemma \ref{volformbound}.
\begin{lemma} \label{alp}
For $0<r<\frac{i_0}{2} \text{ and } M\in \MM$,  let $x\in M \text{ and } B_r(x)$ be a ball of radius r centered at x.
%\begin{enumerate}
    Then $\theta(x)$ has the following bounds \begin{equation}
      1-c_\MMM r\leq \theta(x)\leq 1+c_\MMM r. \nonumber
    \end{equation}
         
 \end{lemma}
The next lemma gives a bound on the $L^2$ norm of the gradient of $\theta$.
\begin{lemma} \label{gradalpha} Let $0<r<\rho<\frac{i_0}{2} \text{ and } M\in \MM$. Let    $\nabla \theta$ denote the gradient of $\theta$.  Then $|\nabla\theta(x)| \leq c_\MMM.$
\end{lemma}
\begin{proof}  Let $A:S^{n-1}\to S^{n-1}$ be the antipodal map defined as $A(u)=-u.$ We note that the Jacobian matrix corresponding to $A$ is $-I$, where $I$ is the identity matrix. By changing co-ordinates using $A$, we can show that
\begin{equation}
    \label{udu=0}
     \int_{u\in S^{n-1}}u du=0.
\end{equation}
One can compute the gradient of $\theta$ in polar coordinates as follows:
\begin{equation}
    \nabla \theta(x)=\frac{n+2}{\nu_n r^{n+2}}\int_{t=0}^r \int_{u\in S^{n-1}}tu J(u,t)t^{n-1}dudt. \nonumber
\end{equation}
\noindent
From equation (\ref{udu=0}), 
\[\int_{u\in S^{n-1}}\int_{t=0}^{r}t^nududt=\int_{u\in S^{n-1}}udu\int_{t=0}^{r}t^ndt=0.\]
Hence,
\begin{equation}
\int_{u\in S^{n-1}}\int_{t=0}^r t J(u,t)t^{n-1}ududt=\int_{u\in S^{n-1}}\int_{t=0}^r t^n (J(u,t)-1)ududt \nonumber   
\end{equation}
From Lemma \ref{volformbound}, we have
\begin{align}
    |J(u,t)-1 |&\leq tc_\MMM.
\end{align}
\par Let $w\in T_xM$, $|w|=1$ be such that $\max\{\langle \triangledown \theta(x),v\rangle| v\in T_xM$, $|v|=1\}$ is attained at $w$. One can then  compute the gradient of $\theta$ in polar coordinates as follows:
\begin{align}
    \lvert \nabla \theta(x)\rvert =\langle \nabla \theta, w\rangle &=\frac{n+2}{\nu_n r^{n+2}}\int_{u\in S^{n-1}} \int_{t=0}^r t^n\langle u, w\rangle \left(J(u,t)-1\right)dtdu \nonumber\\ 
    &\leq\frac{n+2}{\nu_n r^{n+2}}\int_{u\in S^{n-1}}\int_{t=0}^r |\langle u, w\rangle| |1-J(u,t)|t^{n}dtdu \nonumber\\
    &\leq \frac{n+2}{\nu_n r^{n+2}}\int_{u\in S^{n-1}}\int_{t=0}^r c_\MMM t^{n+1}dtdu \nonumber\\
    &\leq c_\MMM .\nonumber
\end{align}
 \end{proof}
Using the associated integral operator and $\theta$, we define the following function.
\begin{definition}   Let $\Lambda_{r}:L^{2}(M)\to C^{0,1}(M)$ be defined as $\Lambda_{r}f=\theta^{-1}\Lambda_{r}^{0}f$.
\end{definition}
We observe that $\Lambda_{r}f=f$, for all constant functions $f.$ In fact, $\Lambda_rf$ is bounded when $r$ is less than the injectivity radius as we see in the following lemma.
\begin{lemma} \label{boundlambda}
Let $0<r<\frac{i_0}{2} \text{ and } M\in \MM$.
    For all $f\in L^{2}(M)$ and $r<\frac{1}{c_{\MMM}}$, we have  
    \begin{equation}
     \norm{\Lambda_{r}f}_{L^{2}}^2\leq \left( \frac{1+c_\MMM r}{1-c_\MMM r}\right)\norm{f}_{L^{2}}^2. \nonumber   
    \end{equation}
 
\end{lemma}

\begin{proof}
    Let $x \in M $. Then \begin{align}
      |\Lambda_{r}^{0}f(x)|^{2} &=|\int_{M}f(y)k_{r}(x,y)dy|^{2} \nonumber \\ &\leq \left(\int_{M}(\sqrt{k_{r}(x,y)})^{2}dy\right)\left(\int_{M}(f(y)\sqrt{k_{r}(x,y)})^{2}dy\right) \label{cs} \\ &=\theta(x)\int_{M}|f(y)|^{2}k_{r}(x,y)dy .\nonumber
    \end{align}
    where (\ref{cs}) follows from Cauchy-Schwarz inequality. \\ 
     So, \begin{align}
        |\Lambda_{r}f(x)|^{2}&=\theta(x)^{-2}|\Lambda_{r}^{0}f(x)|^{2} \nonumber \\ &\leq \theta(x)^{-1}\int_{M}|f(y)|^{2}k_{r}(x,y)dy \nonumber \\ &\leq \lbtheta^{-1}\int_{M}|f(y)|^{2}k_{r}(x,y)dy. \nonumber
    \end{align}
Integrating with respect to the volume measure on  $M$, \begin{align}
   \norm{\Lambda_{r}f}^{2}_{L^{2}}=\int_{M} |\Lambda_{r}f(x)|^{2}dx &\leq  \lbtheta^{-1}\int_{M}|f(y)|^{2}\int_{M}k_{r}(x,y)dydx \nonumber \\ &\leq \lbtheta^{-1}(1+c_\MMM r)\norm{f}^{2}_{L^{2}} \label{alpapp} 
\end{align}
where  (\ref{alpapp}) follows from Lemma \ref{alp}. 
\end{proof}

When $r$ is less than half of the injectivity radius, the following lemma gives an upper bound using average dispersion on how far $\Lambda_r$ takes a function in $L^2(M).$
\begin{lemma} \label{lambda approx} 
Let $0<r<\frac{i_0}{2} \text{ and } M\in\MM$. For all $f \in L^{2}(M)$  and $r<\frac{1}{c_{\MMM}}$,
\begin{equation*}
   \norm{\Lambda_{r}f-f}^2_{L^{2}}\leq \frac{n+2}{\nu_nr^{n}\lbtheta }E_{r}(f). 
\end{equation*}
   
\end{lemma}
\begin{proof}
    We fix a coset representative $f:M\to \mathbb{R}$ of an element of $L^{2}(M).$ 
    \\  Let $a=f(x).$ Now, \begin{align}
        \Lambda_{r}f(x)-f(x)&=\Lambda_{r}f(x)-a \nonumber \\ &=\Lambda_{r}(f-a.\mathbf{1}_{M})(x)  \nonumber \\&=\theta^{-1}(x)\int_{B_{r}(x)}(f(y)-a)k_{r}(x,y)dy  \nonumber\\ &=\theta^{-1}(x)\int_{B_{r}(x)}(f(y)-f(x))k_{r}(x,y)dy. \nonumber
    \end{align}
    By Cauchy-Schwarz inequality,
    \begin{align}
        |\Lambda_{r}f(x)-f(x)|^{2} &\leq \theta^{-2}(x)\left(\int_{B_{r}(x)}k_{r}(x,y)dy\right)\left(\int_{B_{r}(x)}|f(y)-f(x)|^{2}k_{r}(x,y)dy\right) \nonumber \\ &=\theta^{-1}(x)\int_{B_{r}(x)}|f(y)-f(x)|^{2}k_{r}(x,y)dy \nonumber \\ &\leq \frac{n+2}{\nu_{n} r^{n}\lbtheta}\int_{B_{r}(x)}|f(y)-f(x)|^{2}dy \label{kernel bound used}
    \end{align}
     where (\ref{kernel bound used}) follows from Lemma \ref{integral 1}. Finally, by integrating on both sides, we get the result.
\end{proof}

In the following lemma, we obtain an upper bound for the norm of $d\Lambda_r (f)$ using the average dispersion $ E_{r}f$ under Ricci curvature lower bound. 
\begin{lemma} \label{d of lambda r}
Let $0<r<\rho<\frac{i_0}{2} \text{ and } M\in \MM$. For every $f\in L^2(M),$
\begin{equation*}
     \norm{d(\Lambda_{r}f)}^2\leq \frac{n+2}{\nu_{n}r^{n+2}} \frac{(1+ 2c_\MMM r)^2}{(1-c_\MMM r)^3
     } E_{r}f.
\end{equation*}    
\end{lemma}
\begin{proof}
Fixing a coset representative of $L^2(M)$ and denoting it by $f$, let $\Tilde{f}=\Lambda_rf.$ For any $a\in\mathbb{R},$
\begin{equation*}
    \Tilde{f}(x)=a+\theta^{-1}(x)\int_{B_{r}(x)}(f(y)-a)k_{r}(x,y)dy.
\end{equation*}
     Using chain rule and differentiating $\Tilde{f}$ we get, \begin{equation}
       d_{x}\Tilde{f}=\theta^{-1}(x)\int_{B_{r}(x)}(f(y)-a)d_{x}k_{r}(.,y)dy+d_{x}(\theta^{-1})\int_{B_{r}(x)}(f(y)-a)k_{r}(x,y)dy. \nonumber
   \end{equation}
   
 Putting $a=f(x),$ we have \begin{equation} \label{dx norm}
    d_{x}\Tilde{f}=\theta (x)^{-1}A_{1}(x)+A_{2}(x).
\end{equation}  Here, \begin{equation*}
  A_{1}(x)=\int_{B_{r}(x)}(f(y)-f(x))d_{x}k_{r}(.,y)dy  
\end{equation*}and
\begin{equation*}
   A_{2}(x)=d_{x}(\theta^{-1})\int_{B_{r}(x)}(f(y)-f(x))k_{r}(x,y)dy. 
\end{equation*} From Lemma \ref{alp}, and using triangle inequality,
\begin{equation*}
  \norm{d\Tilde{f}}_{L^{2}}\leq \lbtheta^{-1}\norm{A_{1}}_{L^{2}}+\norm{A_{2}}_{L^{2}}.  
\end{equation*} We evaluate each integral $A_1$ and $A_2$ separately to get our desired inequality. 

\par  Using Cauchy-Schwarz inequality, 

 \begin{align} 
 |A_{2}(x)|^{2} &\leq |d_{x}(\theta^{-1})|^{2}\left(\int_{B_{r}(x)}k_{r}(x,y)dy\right)\left(\int_{B_{r}(x)}|f(y)-f(x)|^{2}k_{r}(x,y)dy\right) \nonumber\\
    &\leq |d_{x}(\theta^{-1})|^{2}\theta(x)\int_{B_{r}(x)}|f(y)-f(x)|^{2}k_{r}(x,y)dy \nonumber\\
    &\leq \frac{n+2}{\nu_{n}r^{n}}\frac{|d_{x}\theta|^{2}}{\theta^{3}}\int_{B_{r}(x)}|f(y)-f(x)|^{2}dy\nonumber\\
    &\leq \frac{(n+2)}{\nu_{n}r^{n}}\frac{c_\MMM^2}{\lbtheta^3}\int_{B_{r}(x)}|f(y)-f(x)|^{2}dy.\nonumber   
 \end{align} 

Integrating on $M$, \begin{equation}
    \norm{A_{2}}_{L^{2}}\leq \frac{c_\MMM r}{\lbtheta^\frac{3}{2}}\sqrt{\frac{n+2}{\nu_n r^{n+2}}E_{r}(f)}. \label{A2}
\end{equation} \\ On the other hand, fixing $x\in M,$ let $w\in T_xM$ be such that   
$\text{\normalfont{max}}\{\lvert\langle A_{1}(x),v\rangle\rvert :  v\in T_{x}M, |v|=1\}$ is attained at $w.$ 

So, 
\begin{align}
    \lvert A_1(x) \rvert &= \langle A_1(x), w\rangle \nonumber\\
    &=\frac{n+2}{2\nu_n r^{n+2}}\int_{B_r(x)}\left(f(y)-f(x)\right)\langle \exp^{-1}(y),w\rangle dy \nonumber\\
    &=\frac{n+2}{2\nu_n r^{n+2}}\int_{t=0}^r\int_{u\in S^{n-1}}\left(f(\exp_x(tu))-f(\exp_x(0))\right)\langle tu,w\rangle t^{n-1}du dt. \nonumber
\end{align}
By Cauchy-Schwarz inequality,
\begin{multline*}
    \lvert A_1(x)\rvert^2 \leq \left(\frac{n+2}{2\nu_n r^{n+2}}\right)^2\left(\int_{B_r(0)} \lvert f(\exp_x(tu)-f(\exp_x(0))\rvert^2 J(u,t)^2t^{n-1}dudt\right)\\ \left(\int_{B_r(0)}\langle tu,w\rangle^2t^{n-1}du dt\right). \nonumber
\end{multline*}
Now, $\frac{n+2}{\nu_n r^{n+2}}\int_{B_r(0)}\langle tu,w\rangle^2 t^{n-1}dudt=1.$ Hence, 
\begin{align}
    \lvert A_1(x)\rvert^2&\leq \frac{n+2}{\nu_n r^{n+2}}\int_{B_r(0)}\lvert f(\exp_x(tu))- f(\exp_x(0))\rvert^2 J(u,t)^2 t^{n-1}dudt \nonumber\\ 
    &\leq \frac{n+2}{\nu_n r^{n+2}}(1+c_\MMM r)\int_{B_r(0)}\lvert f(\exp_x(tu))- f(\exp_x(0)) \rvert J(u,t) t^{n-1} du dt \nonumber \\ &= \frac{n+2}{\nu_n r^{n+2}}(1+c_\MMM r)\int_{B_r(x)}\lvert f(y)- f(x)\rvert^2dy. \nonumber
\end{align}

Integrating on $M$, \begin{equation}
\norm{A_{1}}_{L^{2}}\leq \sqrt{\frac{n+2}{\nu_{n}r^{n+2}}(1+c_\MMM r)E_{r}f}.  \label{A1}  
\end{equation}
From (\ref{dx norm}), (\ref{A2}) and (\ref{A1}), \begin{align}
  \norm{df}_{L^{2}}&\leq  \sqrt{\frac{n+2}{\nu_{n}r^{n+2}}}\left(\frac{1+c_\MMM r}{1-c_\MMM r} +\frac{c_\MMM r}{\lbtheta^{\frac{3}{2}}}\right) \sqrt{E_{r}f}\nonumber \\
  &\leq \sqrt{\frac{n+2}{\nu_{n}r^{n+2}}}\frac{1+2 c_\MMM r}{(1-c_\MMM r)^\frac{3}{2}}\sqrt{E_{r}f}. \nonumber
\end{align}
Putting the value of $f_\lambda(r)$ in the above inequality, the lemma follows. 
\end{proof}

\section{A lower bound for $\lambda_k(\Gamma)$}
In this section, we compute an appropriate lower bound for the eigenvalues of $-\Delta_\G$ using the eigenvalues of $-\Delta_M$ for the manifolds in $\MM.$ 

\begin{definition}[Interpolation map] \label{interpolation defn}
 Let $M\in \MM$ and $F$ be a finite $\epsilon$-net in $M$ such that $(F,\mu) \text{ }\epsilon$- approximates  $(M,vol).$  \par Define $I:L^{2}(\G)\to C^{0,1}(M)$ as \begin{equation*}
  Iu=\Lambda_{\rho-2\epsilon}P^{*}u.   
 \end{equation*}  
\end{definition} 
From Lemma \ref{boundlambda} and the fact that $P^*$ preserves the norm, we have \begin{equation*}
 \norm{Iu}_{L^{2}}\leq \left(\frac{1+
 c_\MMM (\rho-2\epsilon)}{1-c_\MMM(\rho-2\epsilon)}\right)\norm{u}_{L^{2}}.  
\end{equation*}
\begin{lemma} \label{interpolation}
  Let $0<\epsilon<\rho<i_0$, $c_\MMM \rho<\frac{1}{4}$ and $\frac{2\epsilon}{\rho}<\frac{1}{n}.$ Consider $M\in \MM$ and $\G $ be a $(\epsilon,\rho)$-approximation  of $M$.  Then
  \begin{enumerate}
        \item $|\norm{Iu}_{L^{2}}-\norm{u}|^2\leq \frac{3\rho^2}{1-c_{\MMM}\rho}   \norm{\delta u}^2.$ 
        
        \item $\norm{dIu}_{L^{2}}\leq \left(1-\frac{2\epsilon}{\rho}\right)^{-\frac{n}{2}-1} \frac{1+ 2c_\MMM \rho}{(1-c_\MMM \rho)^\frac{3}{2}} \norm{\delta u}.$
    \end{enumerate}
\end{lemma}

\begin{proof} (1) Let $u\in L^2(\G).$ Using norm preserving property of $P^*$ and triangle inequality,
    \[\lvert \norm{Iu}_{L^2}-\norm{u}\rvert \leq \norm{Iu-P^*u}_{L^2}.\]
    Using Lemma \ref{lambda approx}, we have 
    \begin{equation}
    \label{i and p*} \norm{Iu-P^*u}_{L^2}^2=\norm{\Lambda_{\rho-2\epsilon}P^*u-P^*u}^2_{L^2}\leq \frac{n+2}{\nu_{n}(1-c_\MMM(\rho-2\epsilon))(\rho-2\epsilon)^{n}}E_{\rho-2\epsilon}(P^*u).   
    \end{equation}
 On the other hand, 
    \[\norm{\delta u}^2=\frac{n+2}{\nu_n \rho^{n+2}}\int_M\int_{U(x)}\lvert P^*u(y)-P^*u(x)\rvert^2dydx,\]
    where $U(x)=\bigcup_{j:x_j\sim x_i}V_j$ such that $x\in V_i$. Since $U(x)\supset B_{\rho-2\epsilon}(x)$,
    \begin{align}
        \label{delta and Er}
        \norm{\delta u}^2&\geq\frac{n+2}{\nu_n \rho^{n+2}}\int_M\int_{B_{\rho-2\epsilon}(x)}\lvert P^*u(y)-P^*u(x)\rvert^2dydx\nonumber\\ &=\frac{n+2}{\nu_n \rho^{n+2}}E_{\rho-2\epsilon}(P^*u).
    \end{align}
    Thus, from (\ref{i and p*}) and (\ref{delta and Er}),
    \begin{equation}
   \label{Iu - P^*} \norm{Iu-P^*u}_{L^2}^2\leq \frac{\rho^{n+2}}{ (1-c_\MMM(\rho-2\epsilon))(\rho-2\epsilon)^n}\norm{\delta u}^2\leq \frac{3\rho^2}{1-c_\MMM\rho}   \norm{\delta u}^2  
\end{equation}
as $\frac{\rho^{n+2}}{(\rho-2\epsilon)^n}<3\rho^2$ for  $\frac{2\epsilon}{\rho}<\frac{1}{n}$.
    \\
(2) Proceeding with direct computation,
    \begin{align}
     \norm{dIu}^{2}_{L^{2}}&=\norm{d(\Lambda_{\rho-2\epsilon}P^{*}u)}^{2}_{L^{2}} \nonumber\\ 
&\leq \frac{n+2}{\nu_{n}(\rho-2\epsilon)^{n+2}}\frac{(1+2 c_\MMM \rho)^2}{(1-c_\MMM \rho)^3} E_{\rho-2\epsilon}(P^*u)  \label{lemm5.10}   \\
     &\leq \left(\frac{\rho}{\rho-2\epsilon}\right)^{n+2}\frac{(1+ 2c_\MMM \rho)^2}{(1-c_\MMM \rho)^3} \norm{\delta u}^{2} \label{e approx}  \end{align}
    where (\ref{lemm5.10}) follows from Lemma \ref{d of lambda r} and (\ref{e approx}) follows from (\ref{delta and Er}).   
    \end{proof}
Using the above lemma and the results obtained on the interpolation map, we are now ready to compute, under certain constraints, a lower bound for the eigenvalue of the graph Laplacian operator using the eigenvalue of the Laplace-Beltrami operator for all $M\in\MM$. Since $\LL$ is an upper bound for all $k$-th eigenvalues without loss of generality, we consider $\LL> 3(1+4c_\MMM).$
\begin{theorem} \label{lowbd}
For $0<\epsilon<\rho<\frac{i_0}{2}$, let $M\in \MM, $ and $\G=(F,\mu)$ be an $(\epsilon,\rho)$- approximation of $(M,\normalfont{\text{vol}}).$ Let $\lk(M)$ and $\lambda_{k}(\Gamma)$ denote the k-th eigenvalues of $-\Delta_{M}$  and $-\Delta_{\G}$ respectively. There exist positive constants $C_n$ and $C_{\MMM}$ depending solely on $n$ and $\MMM$, respectively, such that  
    \begin{equation}
       \lk(\G) \geq\left(1-C_n \frac{\epsilon}{\rho}-C_\MMM\rho\right)\lambda_k(M)-C_{\MMM}\rho\lk^{\frac{3}{2}}(M). \nonumber
    \end{equation}
    
\end{theorem}
\begin{proof} If $\lambda_k(\Gamma)>\lambda_k(M)$ then the theorem follows immediately. So, we assume that  $\lambda_k(\Gamma)\leq\lambda_k(M)$. We use the min-max principle again to prove the above inequality. Therefore, it is enough to show there is a subspace $L\subset H^{1}(M)$ with dim$L=k$ such that for any non-zero $f\in L$,  $\frac{\|df\|_{L^2}^2}{\|f\|_{L^2}^2}$ is less than or equal to the right hand side of the above inequality. We consider the $ k$-dimensional subspace $W$ of $L^2(\G)$ which is the linear span of the first $k$ orthonormal eigenfunctions of $-\Delta_{\Gamma}$. Let $u\in W$. Then we have $\norm{\delta u}^{2}\leq \lambda_{k}(\Gamma)\norm{u}^{2}.$ Using Lemma \ref{interpolation}(1), \begin{equation}
        \label{deno} \norm{Iu}_{L^{2}}\geq \left(1-\rho\sqrt{\frac{3\lambda_k(\Gamma)}{1-c_{\MMM}\rho}}\right)\norm{u}. \nonumber
    \end{equation}
    Let $f=Iu.$ As $\lambda_k(\Gamma)\leq \lambda_{k}(M)$  we have,
    \begin{equation} \label{nemo}
         \norm{f}^{2}_{L^{2}} \geq  \left(1-\rho\sqrt{\frac{3\lambda_k(M)}{1-c_{\MMM}\rho}}\right)^2\norm{u}^2  
        \end{equation}
From Lemma \ref{interpolation}(2),
\begin{align}
 \norm{df}^{2}_{L^{2}}&\leq  \left(1-\frac{2\epsilon}{\rho}\right)^{-n-2} \frac{(1+ 2c_\MMM \rho)^2}{(1-c_\MMM \rho)^3}\norm{\delta u}^{2} \nonumber \\
        &\leq  \left(1-\frac{2\epsilon}{\rho}\right)^{-n-2} \frac{(1+ 2c_\MMM \rho)^2}{(1-c_\MMM \rho)^3}\lambda_k(\Gamma)\norm{u}^2.      
\end{align}
From equations (\ref{nemo}) and (\ref{deno}), we have
    \begin{align}
        \lambda_k(M)\leq\frac{\norm{df}^{2}}{\norm{f}^{2}}&\leq \left(1-\frac{2\epsilon}{\rho}\right)^{-n-2}\frac{(1+ 2c_\MMM \rho)^2}{(1-c_\MMM \rho)^2(1-c_{\MMM}\rho-4\rho\sqrt{(1-\rho c_\MMM)\lk(M)}) } \lambda_k(\Gamma).\nonumber
    \end{align}
For $\rho<\frac{1}{4c_\MMM}$ and $\frac{2\epsilon}{\rho}<1$ we have, 
 \begin{align}
        \lambda_k(\Gamma) &\geq (1-\frac{2\epsilon}{\rho})^{n+2}(1-c_\MMM \rho)^2(1-c_{\MMM}\rho-4\rho\sqrt{(1-\rho c_\MMM)\lk(M)})(1+ 2c_\MMM \rho)^{-2}\lambda_k(M)\\
        &\geq \left(1-2(n+2)\frac{\epsilon}{\rho}\right)(1-2c_\MMM \rho)\left(1-c_{\MMM}\rho-4\rho\sqrt{\lk(M)}\right)(1+5c_{\MMM}\rho)^{-1}\lambda_k(M)\nonumber
    \end{align}
  Now the theorem follows from the Taylor series expansion of $(1+5c_{\MMM}\rho)^{-1}.$ 
    \end{proof}

\subsection*{Proof of Theorem \ref{main 1} :} The proof follows from Theorem \ref{upperbd} and Theorem \ref{lowbd}.

\section{Approximation of eigenfunctions}
In this section, we investigate the approximation of eigenfunctions of $M\in\MM$ using similar arguments as in \cite{article}. The following lemma shows that the discretization map and the interpolation map are almost inverses of each other. 
\begin{lemma} \label{almost inverses}\begin{enumerate}
      \item   Let $f\in H^1(M)$. Then for $\rho<\frac{1}{2c_\MMM}$,
   \[\norm{IPf-f}\leq C_n\rho\norm{df}.\]
   
  \item Let $u\in L^2(\G) $, $\rho\leq \frac{1}{3c_\MMM}$ and $\frac{\epsilon}{\rho}<\frac{1}{2n}$. Then,
   \[\norm{PIu-u}\leq C_n\rho\norm{\delta u}\]
  \end{enumerate}
  where $C_n$ is a positive constant depending only on the dimension $n.$
  \end{lemma}
 \begin{proof} 
  \begin{enumerate}

     \item By straightforward computation,
      \begin{align}
          \norm{IPf-f}&=\norm{\Lambda_{\rho-2\epsilon}P^*Pf-f} \nonumber\\
          &\leq \norm{\Lambda_{\rho-2\epsilon}(P^*Pf-f)} +\norm{\Lambda_{\rho-2\epsilon}f-f}. \label{traiglineq}
      \end{align}
      We aim to find upper bounds of both the terms obtained above using the triangle inequality.
      From Lemma \ref{boundlambda} and Lemma \ref{P*Pf-f},
      \begin{equation}  \label{pt1}
       \norm{\Lambda_{\rho-2\epsilon}(P^*Pf-f)}^2\leq \frac{4n\nu_{n}}{c_n }\frac{(1+c_\MMM \rho)^2}{1-c_\MMM\rho}\epsilon^2\norm{df}_{L^{2}}^2. 
      \end{equation}
      Again from Lemma \ref{lambda approx} and Lemma \ref{theo1},
      \begin{equation}
          \label{pt2}
          \norm{\Lambda_{\rho-2\epsilon}f-f}^2\leq \frac{n+2}{\nu_{n}(1-c_\MMM\rho)
          (\rho-2\epsilon)^n}E_{\rho-2\epsilon}(f)\leq \frac{1+c_\MMM \rho}{1-c_\MMM\rho}
          \rho^{2}\norm{df}^2.
      \end{equation}

 From equations (\ref{traiglineq}), (\ref{pt1}) and (\ref{pt2}), we have, 
\[
\norm{IPf-f}\leq \frac{(1+c_\MMM \rho)}{\sqrt{1-c_{\MMM}\rho}}\left(1+\sqrt{\frac{n\nu_n}{c_n}}\right)\rho\norm{df}. \]
As $c_{\MMM}\rho <\frac{1}{2}$, the first part of the lemma follows.
\item
Since $P^*$ preserves norm, using the triangle inequality
\begin{align}
    \norm{PIu-u} = \norm{P^*(PIu-u)} \leq \norm{P^*PIu-Iu} + \norm{Iu-P^*u}.
\end{align}

Using Lemma \ref{P*Pf-f}, for $\rho<\frac{1}{c_\MMM}$,
\[\norm{P^*PIu-Iu}^2\leq \frac{4n\nu_n}{c_n}(1+c_\MMM \rho)\epsilon^2\norm{dIu}^2.\]
By Lemma \ref{i and p*}  we have,
\[\norm{dIu}^2_{L^{2}}\leq  (1-\frac{2\epsilon}{\rho})^{-n-2}\frac{(1+2c_\MMM \rho)^2}{(1-c_\MMM\rho)^3} \norm{\delta u}^{2}\]
Therefore, for $\rho< \frac{1}{c_\MMM}$ and $\frac{2\epsilon}{\rho}<\frac{1}{n}$,
\begin{equation}
\label{P*PIu-Iu}
 \norm{P^*PIu-Iu}\leq 2\sqrt{\frac{3n\nu_n}{c_n}}\frac{(1+2c_{\MMM}\rho)^{\frac{3}{2}}}{(1-c_\MMM \rho)^{\frac{3}{2}}} \epsilon\norm{\delta u}.  
\end{equation}
From equation (5.3) we have,
\begin{equation}
    \label{Iu-P*u 2}
    \norm{Iu-P^*u}\leq \frac{3\rho}{1-2c_\MMM\rho}\norm{\delta u}.
\end{equation}

Hence from (\ref{Iu-P*u 2}) and (\ref{P*PIu-Iu}),
\[\norm{PIu-u}\leq 3 \left( 1+\sqrt{\frac{\nu_n}{c_n}}\right)\frac{(1+2c_{\MMM}\rho)^{\frac{3}{2}}}{(1-2c_\MMM \rho)^{\frac{3}{2}}}\rho\norm{\delta u}.\]
 \end{enumerate}
 The second part of the lemma follows from the fact that $c_{\MMM}\rho<\frac{1}{3}.$
  \end{proof}
   Let $I\subset \mathbb{R}$ be an interval, and define $H_I(M)$(respectively $H_I(X)$) as the subspace of $H^1(M)$ (respectively $H^1(X)$)  spanned by eigenfunctions corresponding to eigenvalues in $I.$  We denote $H_{(-\infty,\lambda)}(M)$ by $H_\lambda (M).$ Let $\mathbb{P}_I: L^2(M)\to H_I(M)$ be the orthogonal projection. Analogous notation is used for projections from $L^2(\G)$ to $H_J(X)$ and we write  $\mathbb{P}_\lambda$ as shorthand for $\mathbb{P}_{(-\infty,\lambda)}$. 
  \begin{lemma} \label{almost isometry}
      \begin{enumerate}
          \item Let $\lambda>0$ and $f\in H_\lambda(M).$ Then
          \[\norm{\delta(Pf)}\geq \left(1-\sigma \right)\norm{df}\]
         where $\sigma=C_\MMM\rho + C_n\sqrt{\lambda}\rho+C_n\frac{\epsilon}{\rho}.$
          \item Let $\lambda>0$ and $u\in H_\lambda(X).$ Then
          \[ \norm{d(Iu)}\geq \left(1-\sigma\right)\norm{\delta u}\]
        where $\sigma=C_\MMM\rho + C_n\sqrt{\lambda}\rho+C_n\frac{\epsilon}{\rho}.$  
      \end{enumerate}
  \end{lemma}
  \begin{proof}
     \begin{enumerate}
         \item As $\mathbb{P}_\lambda$ is non-increasing with respect to the Dirichlet energy,
         \[\norm{d(IPf)}_{L^2}\geq \norm{d(\mathbb{P}_\lambda IPf)}_{L^2}\geq \norm{df}_{L^2}-\norm{d(\mathbb{P}_\lambda IPf-f)}_{L^2}.\]
         Given that  $f\in H_\lambda(M)$, Lemma  \ref{almost inverses} implies for $\rho<\frac{1}{3c_\MMM}$,      \begin{align}
            \norm{d(\mathbb{P}_\lambda IPf-f)}_{L^2}  \leq \sqrt{\lambda}\|\mathbb{P}_\lambda (IPf-f)\| &\leq C_n\sqrt{\lambda}\rho\norm{df}.\nonumber
        \end{align}
      
    Hence, 
    \begin{equation}
    \label{dIpf}
    \norm{d(IPf)}_{L^2}\geq \left(1-C_n\sqrt{\lambda}\rho\right)\norm{df}.    
    \end{equation}
    For $\rho<\frac{1}{4c_\MMM}$ and $\frac{\epsilon}{\rho}<\frac{1}{2n}$, from Lemma \ref{interpolation} there exists a constant $C_n, C_{\MMM}>0$ such that
    \begin{equation}
        \label{dipf2}
        \norm{d(IPf)} \leq (1+C_n\frac{\epsilon}{\rho}+C_{\MMM} \rho)\norm{\delta Pf}.
    \end{equation}
    Thus, from (\ref{dIpf}) and (\ref{dipf2}),
    \begin{align}
    \norm{\delta(Pf)} &\geq \left(1-\sigma \right)\norm{df}\nonumber
    \end{align}
 where $\sigma=C_\MMM\rho + C_n\sqrt{\lambda}\rho+C_n\frac{\epsilon}{\rho}.$

     \item Applying the techniques from the previous part and Lemma \ref{almost inverses}, for $\rho<\frac{1}{3 c_\MMM }$ we obtain

    \begin{equation}
        \label{deltaIu2}
        \norm{\delta(PIu)}_{L^2}\geq \left(1- C_n\sqrt{\lambda}\rho\right)\norm{\delta u}.
    \end{equation}
     
     From Lemma \ref{delp}, for  $\frac{\epsilon}{\rho}<\frac{1}{C_n}$
    \begin{equation}
        \label{deltaIu}
        \norm{\delta(PIu)}\leq(1+C_n\frac{\epsilon}{\rho}+C_{\MMM}\rho)\norm{d(Iu)}.    
    \end{equation}
    From (\ref{deltaIu2}) and (\ref{deltaIu}),
    \begin{align}
            \norm{d(Iu)}&\geq \frac{1- C_n\sqrt{\lambda}\rho}{(1+C_n\frac{\epsilon}{\rho}+C_{\MMM}\rho)}\norm{\delta u} \nonumber \\
        &\geq (1-\sigma)\norm{\delta u} \nonumber
       \end{align}
    where $\sigma=C_\MMM\rho + C_n\sqrt{\lambda}\rho+C_n\frac{\epsilon}{\rho}.$

         \end{enumerate}
  \end{proof}
  
  Let $\{f_k\}_1^\infty$ and $\{u_k\}_1^N$ be orthonormal eigenvectors of $\Delta_M$ and $\Delta_X$ respectively.
The next lemma approximates eigenfunctions of $\Delta_M$, which has been proved using Lemma 7.3 from \cite{article} with analogous inequalities and approximations in our setting. 
  \begin{lemma} \label{7.3}Let $0<\epsilon<\rho<\frac{i_0}{2}$, $k\in\mathbb{N}$ and $a\in \mathbb{R}_+$ 
      \begin{enumerate}
          \item Let $\lambda=\lambda_k(M).$ Then
          \[\norm{Pf_k-\mathbb{P}_{\lambda_{k+a}}Pf_k}^2\leq \frac{1}{a}C_{\MMM,k}\left(\frac{\epsilon}{\rho}+\rho \right).\]
         
          Also \[\norm{\delta(Pf_k-\mathbb{P}_{\lambda+a}Pf_k)}^2\leq \CC\left(1+\frac{1}{a}\right)\left(\frac{\epsilon}{\rho}+\rho \right).\]
          
          \item Let $\lambda=\lambda_k(X).$ Then 
          \[\norm{Iu_k-\mathbb{P}_{\lambda+a}Iu_k}^2\leq \frac{1}{a}C_{\MMM,k}\left(\frac{\epsilon}{\rho}+\rho \right).\]
          
          Also \[\norm{d(Iu_k-\mathbb{P}_{\lambda+a}Iu_k}^2\leq \CC\left(1+\frac{1}{a}\right)\left(\frac{\epsilon}{\rho}+\rho \right).\]
      \end{enumerate}
  \end{lemma}
  \begin{proof}
  \begin{enumerate}
      \item From the proof of Theorem \ref{uppbd},  we have dim$L$= $k$ if  $\epsilon<\frac{1}{C_{\MMM,k}}$. Let $Q$ denote the discrete Dirichlet energy form on $L^2(\G)$ and $\{\lambda_j^L:1\leq j\leq k\}$ be the eigenvalues of $Q$ restricted to $L.$ \\
For $f\in W$, from Corollary \ref{cor2}
\begin{equation}
    \left(1-C_n\left(\frac{\epsilon}{\rho}+c_\MMM\rho\right)\right)\norm{f}\leq \norm{Pf} \leq \left(1+C_n\left(\frac{\epsilon}{\rho}+c_\MMM\rho\right)\right)\norm{f} \label{pfk bound}
\end{equation}
 and
\begin{equation}
    \norm{\delta(Pf)}\leq C_n\left(\frac{\epsilon}{\rho}+c_\MMM\rho\right)\norm{df}. \nonumber
\end{equation}

By minimax principle, for $\rho<\frac{1}{C_nc_\MMM}$,
\begin{equation}
    \lambda_j^L\leq \lambda_j(M) + \LL\left( 1+C_n\left(\frac{\epsilon}{\rho}+c_\MMM\rho\right)\right) \label{lambda bound quad}
\end{equation}
Let $Q':L^2(\G)\rightarrow \mathbb{R}$ be defined as
\[Q'(u)=Q(\mathbb{P}_{\lambda+a}(u))+\lambda\norm{u-\mathbb{P}_{\lambda+a}(u)}^2.\]
We observe that $Q' $ is a quadratic form and $Q'\leq Q.$ Hence, for every $j\leq m$ and $V\subset L^2(\G)$ such that dim$V=j$,
\[\sup_{v\in V\setminus\{0\}}\frac{Q'(v)}{\norm{v}^2}\geq \min\{\lambda,\lambda_j(\G)\}.\]
Let $\{\lambda_j':1\leq j\leq k\}$ be the eigenvalues of $Q'|_L.$ Using the above equality and the minimax principle $\lambda'_j\geq\min\{\lambda,\lambda_j(\G)\}$ for all $j\leq k.$ 

By Theorem \ref{main 1},
\begin{equation}
    \lambda'_j\geq \lambda_j(M)- C_{\MMM,k}\left(\frac{\epsilon}{\rho}+\rho \right). \nonumber
\end{equation}
From (\ref{lambda bound quad}),
\begin{equation}
   \lambda_j^L-\lambda_j'\leq  C_{\MMM,k}\left(\frac{\epsilon}{\rho}+\rho \right) \nonumber
\end{equation}
for all $j\leq k.$
Hence, for every $u\in L$,
\begin{equation}
   Q(u)-Q'(u)\leq C_{\MMM,k}\left(\frac{\epsilon}{\rho}+\rho \right). \label{max diff of eigenvalues} 
\end{equation}
Let $u'=u-\mathbb{P}_{\lambda+a}u.$ Since $Q(u')\geq (\lambda+a)\norm{u'}^2,$
\begin{equation}
    Q(u)-Q'(u) = Q(u)-Q'(u') \geq \frac{a}{\lambda+a}Q(u'). \label{q and q'}
\end{equation}
From (\ref{max diff of eigenvalues}) and (\ref{q and q'}), 
\[Q(u')\leq \frac{\lambda+a}{a}C_{\MMM,k}\left(\frac{\epsilon}{\rho}+\rho \right)\norm{u}^2.\]
Hence,
\[ \norm{u'}^2\leq \frac{1}{a}C_{\MMM,k}\left(\frac{\epsilon}{\rho}+\rho \right)\norm{u}^2.\]
Substituting $u=Pf_k$, and observing from (\ref{pfk bound}) that for sufficiently small $\epsilon, \rho, \frac{\epsilon}{\rho}$, $\norm{Pf_k}<2$, we obtain the required expression.
\item The second assertion follows analogously using Lemma
\ref{interpolation} and appropriate assumptions for $c_\MMM$ and $\LL$ as above.   \end{enumerate}   \end{proof}
The following is a generalization of  Lemma 7.4 from \cite{article}. 
  \begin{lemma} \label{- gamma +alpha projection}
  Let $ \epsilon<\C^{-1}$ and $\C>1$. 
   \begin{enumerate}
       \item Let $\lambda=\lambda_k(M)$ such that $0<\alpha\leq \beta\leq \gamma\leq 1$ and there are no eigenvalues of $-\Delta_\Gamma$ in $(\lambda+\alpha,\lambda+\beta).$ Then 
       \[\norm{Pf_k-\mathbb{P}_{(\lambda-\gamma,\lambda+\alpha]}Pf_k}^2\leq \CC\alpha\gamma^{-1}+\CC\beta^{-1}\gamma^{-1}\left(\frac{\epsilon}{\rho}+\rho \right).\]
       \item Let $\lambda=\lambda_k(\Gamma)$ and $\alpha,\beta,\gamma>0$ such that $\alpha\leq \beta\leq \gamma\leq 1$ and there are no eigenvalues of $-\Delta_M$ in $(\lambda+\alpha,\lambda+\beta).$ Then 
       \[\norm{Iu_k-\mathbb{P}_{(\lambda-\gamma,\lambda+\alpha]}Iu_k}^2\leq \CC\alpha\gamma^{-1}+\CC\beta^{-1}\gamma^{-1}\left(\frac{\epsilon}{\rho}+\rho \right).\]
   \end{enumerate}   
  \end{lemma}
\begin{proof}
    \begin{enumerate}
        \item Let $Q$ denote the discrete Dirichlet energy on $L^2(\G)$ and $u=Pf_k$. We can express $u$ as a decomposition of orthogonal vectors \[u=u_0+u_-+u_+\] where $u_0\in H_{(\lambda-\gamma,\lambda+\alpha)}(X)$, $u_-\in H_{(-\infty,\lambda-\gamma]}(X)$ and $u_+\in H_{[\lambda+\alpha,\infty)}(X)$. From Lemma \ref{7.3},
        \begin{equation} \label{u_+ bound}
            \norm{u_+}^2 \leq \frac{1}{\beta}C_{\MMM,k}\left(\frac{\epsilon}{\rho}+\rho \right)
        \end{equation}
        and the Dirichlet energy for is bounded above by
        \begin{equation}
            Q(u_+)\leq  \frac{1}{\beta}C_{\MMM,k}\left(\frac{\epsilon}{\rho}+\rho \right).\nonumber
        \end{equation}
        Using appropriate bounds on $c_\MMM$ and $\LL$ from Lemma \ref{almost isometry},
        \[Q(u)\geq \left(1-C_{\MMM,k}\left(\frac{\epsilon}{\rho}+\rho \right)\right)\lambda.\]
        Hence, 
        \[Q(u_0)+Q(u_-)=Q(u)-Q(u_+)\geq \lambda-\frac{1}{\beta}C_{\MMM,k}\left(\frac{\epsilon}{\rho}+\rho \right).\]
    
    By the min-max principle, 
    \[Q(u_0)\leq (\lambda+\alpha)\norm{u_0}^2\]
    and
    \[Q(u_-)\leq (\lambda-\gamma)\norm{u_-}^2,\]
    Thus,
    \[\lambda- \frac{1}{\beta}C_{\MMM,k}\left(\frac{\epsilon}{\rho}+\rho \right)\leq \lambda(\norm{u_0}^2+\norm{u_-}^2)+\alpha\norm{u_0}^2\gamma\norm{u_-}^2.\]
    Also, from (\ref{pfk bound}),
    \[\norm{u_0}^2\leq \norm{Pf_k}^2\leq 1+C_{\MMM,k}\left(\frac{\epsilon}{\rho}+\rho \right).\]
    Hence, \[\lambda- \frac{1}{\beta}C_{\MMM,k}\left(\frac{\epsilon}{\rho}+\rho \right)\leq\lambda(1+\sigma)+\alpha(1+\sigma)-\gamma\norm{u_-}^2\]
    where $\sigma=C_{\MMM,k}\left(\frac{\epsilon}{\rho}+\rho \right)$.
    Hence for $\rho<\frac{\alpha\gamma}{\CC}$, \[ \norm{u_-}^2\leq \beta^{-1}\gamma^{-1}C_{\MMM,k}\left(\frac{\epsilon}{\rho}+\rho \right)+\alpha\gamma^{-1}(1+\CC)\]
    Using (\ref{u_+ bound}), we derive the required assertion. 
    \end{enumerate}
    \item Following similar methods as in the first part and taking appropriate bounds for $c_\MMM$ and $\LL$, the second assertion can be obtained.
\end{proof} 
  
  \begin{theorem}
   \begin{enumerate}
       \item Let $\lambda=\lambda_k(M)$ and let $f_k$ be corresponding unit-norm eigenfunction of $\Delta_M.$ Then for every $\gamma\in (0,1)$,
       \[\norm{Pf_k-\mathbb{P}_{(\lambda-\gamma,\lambda+\gamma)}Pf_k}^2\leq \CC\gamma^{-2}\left(\frac{\epsilon}{\rho}+\rho \right).\]
       \item Let $\lambda=\lambda_k(M)$ and let $f_k$ be corresponding unit-norm eigenfunction of $\Delta_M.$ Then for every $\gamma\in (0,1)$,
       \[\norm{Iu_k-\mathbb{P}_{(\lambda-\gamma,\lambda+\gamma)}Iu_k}^2\leq \CC\gamma^{-2}\left(\frac{\epsilon}{\rho}+\rho \right).\]
   \end{enumerate}   
  \end{theorem}

\begin{proof}
    Substituting $\alpha=\beta=C_{\MMM,k}\left(\frac{\epsilon}{\rho}+\rho \right)^\frac{1}{2}\gamma$ in Lemma \ref{- gamma +alpha projection} and using the fact that $\gamma^{-2}\geq 1$ we get the required expression.
\end{proof}
  
  \begin{theorem}
      Let $\lambda=\lambda_j(M)$ be an eigenvalue of $\Delta_M$ with multiplicity $m$, such that
$$ \lambda_{k-1} < \lambda_k = \lambda = \lambda_{k+m-1} < \lambda_{k+m}. $$
 Let $\delta_\lambda = \min\{1, \lambda_k - \lambda_{k-1}, \lambda_{k+m} - \lambda_{k+m-1}\}$.
Let $u_k, \dots, u_{k+m-1}$ be orthonormal eigenvectors  corresponding to  $\lambda_k(\Gamma), \dots, \lambda_{k+m-1}(\Gamma)$.

Then there exist orthonormal eigenfunctions $g_k, \dots, g_{k+m-1}$ of $-\Delta_M$ corresponding to $\lambda$ such that for all $j = k, \dots, k+m-1$ and $ \rho<\delta_\lambda C_{\mathcal{M}, k}^{-1}$,
\begin{equation}
\norm{u_j - P g_j}^2 \leq C_{\mathcal{M}, k} \delta_\lambda^{-2} \left(\frac{\epsilon}{\rho}+\rho \right)
\end{equation}
and
\begin{equation}
\norm{g_j - I u_j}^2 \leq C_{\mathcal{M}, k} \delta_\lambda^{-2}\left(\frac{\epsilon}{\rho}+\rho \right).
\end{equation}
\end{theorem}
  \begin{proof}
   The proof of the above theorem is analogous to that of Theorem 4 in \cite{article}, utilizing the results of Lemma \ref{- gamma +alpha projection} and Lemma \ref{almost isometry}. 
  \end{proof}

 \bibliographystyle{plain} 
\bibliography{bib}

@article {MR3990939,
    AUTHOR = {Burago, Dmitri and Ivanov, Sergei and Kurylev, Yaroslav},
     TITLE = {Spectral stability of metric-measure {L}aplacians},
   JOURNAL = {Israel J. Math.},
  FJOURNAL = {Israel Journal of Mathematics},
    VOLUME = {232},
      YEAR = {2019},
    NUMBER = {1},
     PAGES = {125--158},
      ISSN = {0021-2172,1565-8511},
   MRCLASS = {53C23 (58J50)},
  MRNUMBER = {3990939},
MRREVIEWER = {He-Jun\ Sun},
       DOI = {10.1007/s11856-019-1865-7},
       URL = {https://doi.org/10.1007/s11856-019-1865-7},
}

@article {MR1975337,
    AUTHOR = {Petrunin, Anton},
     TITLE = {Polyhedral approximations of {R}iemannian manifolds},
   JOURNAL = {Turkish J. Math.},
  FJOURNAL = {Turkish Journal of Mathematics},
    VOLUME = {27},
      YEAR = {2003},
    NUMBER = {1},
     PAGES = {173--187},
      ISSN = {1300-0098,1303-6149},
   MRCLASS = {53C20},
  MRNUMBER = {1975337},
MRREVIEWER = {Igor\ Belegradek},
}

@article {MR2130531,
    AUTHOR = {Mantuano, Tatiana},
     TITLE = {Discretization of compact {R}iemannian manifolds applied to
              the spectrum of {L}aplacian},
   JOURNAL = {Ann. Global Anal. Geom.},
  FJOURNAL = {Annals of Global Analysis and Geometry},
    VOLUME = {27},
      YEAR = {2005},
    NUMBER = {1},
     PAGES = {33--46},
      ISSN = {0232-704X,1572-9060},
   MRCLASS = {58J50 (53C20)},
  MRNUMBER = {2130531},
MRREVIEWER = {Chadwick\ Sprouse},
       DOI = {10.1007/s10455-005-5215-0},
       URL = {https://doi.org/10.1007/s10455-005-5215-0},
}

@article {MR488179,
    AUTHOR = {Dodziuk, J. and Patodi, V. K.},
     TITLE = {Riemannian structures and triangulations of manifolds},
   JOURNAL = {J. Indian Math. Soc. (N.S.)},
  FJOURNAL = {The Journal of the Indian Mathematical Society. New Series},
    VOLUME = {40},
      YEAR = {1976},
    NUMBER = {1-4},
     PAGES = {1--52},
      ISSN = {0019-5839,2455-6475},
   MRCLASS = {58G10 (57D20)},
  MRNUMBER = {488179},
MRREVIEWER = {M.\ L.\ Gromov},
}

@article {MR4411102,
    AUTHOR = {Lu, Jinpeng},
     TITLE = {Graph approximations to the {L}aplacian spectra},
   JOURNAL = {J. Topol. Anal.},
  FJOURNAL = {Journal of Topology and Analysis},
    VOLUME = {14},
      YEAR = {2022},
    NUMBER = {1},
     PAGES = {111--145},
      ISSN = {1793-5253,1793-7167},
   MRCLASS = {58J50 (05C50 53C21 53C23 58J60 65N25)},
  MRNUMBER = {4411102},
MRREVIEWER = {Shiping\ Liu},
       DOI = {10.1142/S1793525320500442},
       URL = {https://doi.org/10.1142/S1793525320500442},
}

@article {MR4437353,
    AUTHOR = {Burago, Dmitri and Ivanov, Sergei and Kurylev, Yaroslav and
              Lu, Jinpeng},
     TITLE = {Approximations of the connection {L}aplacian spectra},
   JOURNAL = {Math. Z.},
  FJOURNAL = {Mathematische Zeitschrift},
    VOLUME = {301},
      YEAR = {2022},
    NUMBER = {3},
     PAGES = {3185--3206},
      ISSN = {0025-5874,1432-1823},
   MRCLASS = {58C40 (53C21 58J60 65J10)},
  MRNUMBER = {4437353},
MRREVIEWER = {He-Jun\ Sun},
       DOI = {10.1007/s00209-022-03016-5},
       URL = {https://doi.org/10.1007/s00209-022-03016-5},
}

@article {MR4620352,
    AUTHOR = {Bl\aa sten, Emilia and Isozaki, Hiroshi and Lassas, Matti and
              Lu, Jinpeng},
     TITLE = {Gelfand's inverse problem for the graph {L}aplacian},
   JOURNAL = {J. Spectr. Theory},
  FJOURNAL = {Journal of Spectral Theory},
    VOLUME = {13},
      YEAR = {2023},
    NUMBER = {1},
     PAGES = {1--45},
      ISSN = {1664-039X,1664-0403},
   MRCLASS = {35R02 (05C50 35J05 52C25)},
  MRNUMBER = {4620352},
       DOI = {10.4171/jst/455},
       URL = {https://doi.org/10.4171/jst/455},
}

@article {MR4695859,
    AUTHOR = {Colbois, Bruno and Girouard, Alexandre and Gordon, Carolyn and
              Sher, David},
     TITLE = {Some recent developments on the {S}teklov eigenvalue problem},
   JOURNAL = {Rev. Mat. Complut.},
  FJOURNAL = {Revista Matem\'atica Complutense},
    VOLUME = {37},
      YEAR = {2024},
    NUMBER = {1},
     PAGES = {1--161},
      ISSN = {1139-1138,1988-2807},
   MRCLASS = {58C40 (35P05 35P15 35P20 53A10 58J50)},
  MRNUMBER = {4695859},
       DOI = {10.1007/s13163-023-00480-3},
       URL = {https://doi.org/10.1007/s13163-023-00480-3},
}

@article{article,
    AUTHOR = {Burago, Dmitri and Ivanov, Sergei and Kurylev, Yaroslav},
     TITLE = {A graph discretization of the {L}aplace-{B}eltrami operator},
   JOURNAL = {J. Spectr. Theory},
  FJOURNAL = {Journal of Spectral Theory},
    VOLUME = {4},
      YEAR = {2014},
    NUMBER = {4},
     PAGES = {675--714},
      ISSN = {1664-039X,1664-0403},
   MRCLASS = {58J50 (05C50 53C21 58J60 65N25)},
  MRNUMBER = {3299811},
MRREVIEWER = {Emil\ Saucan},
       DOI = {10.4171/JST/83},
       URL = {https://doi.org/10.4171/JST/83},
}

@article {MR1257106,
    AUTHOR = {Fujiwara, Koji},
     TITLE = {Eigenvalues of {L}aplacians on a closed {R}iemannian manifold
              and its nets},
   JOURNAL = {Proc. Amer. Math. Soc.},
  FJOURNAL = {Proceedings of the American Mathematical Society},
    VOLUME = {123},
      YEAR = {1995},
    NUMBER = {8},
     PAGES = {2585--2594},
      ISSN = {0002-9939,1088-6826},
   MRCLASS = {58G25 (58G99)},
  MRNUMBER = {1257106},
MRREVIEWER = {Robert\ Brooks},
       DOI = {10.2307/2161293},
       URL = {https://doi.org/10.2307/2161293},
}

@article {MR4130541,
    AUTHOR = {Garc\'ia Trillos, Nicol\'as and Gerlach, Moritz and Hein,
              Matthias and Slep\v cev, Dejan},
     TITLE = {Error estimates for spectral convergence of the graph
              {L}aplacian on random geometric graphs toward the
              {L}aplace-{B}eltrami operator},
   JOURNAL = {Found. Comput. Math.},
  FJOURNAL = {Foundations of Computational Mathematics. The Journal of the
              Society for the Foundations of Computational Mathematics},
    VOLUME = {20},
      YEAR = {2020},
    NUMBER = {4},
     PAGES = {827--887},
      ISSN = {1615-3375,1615-3383},
   MRCLASS = {62G20 (05C50 58J50 60D05 62R30 65N25 68R10 81Q35)},
  MRNUMBER = {4130541},
       DOI = {10.1007/s10208-019-09436-w},
       URL = {https://doi.org/10.1007/s10208-019-09436-w},
}

@book {petersen,
    AUTHOR = {Petersen, Peter},
     TITLE = {Riemannian geometry},
    SERIES = {Graduate Texts in Mathematics},
    VOLUME = {171},
   EDITION = {Second},
 PUBLISHER = {Springer, New York},
      YEAR = {2006},
     PAGES = {xvi+401},
      ISBN = {978-0387-29246-5; 0-387-29246-2},
   MRCLASS = {53-01 (53C20 53C21 53C23)},
  MRNUMBER = {2243772},
}

@article {MR4273695,
    AUTHOR = {Wormell, Caroline L. and Reich, Sebastian},
     TITLE = {Spectral convergence of diffusion maps: improved error bounds
              and an alternative normalization},
   JOURNAL = {SIAM J. Numer. Anal.},
  FJOURNAL = {SIAM Journal on Numerical Analysis},
    VOLUME = {59},
      YEAR = {2021},
    NUMBER = {3},
     PAGES = {1687--1734},
      ISSN = {0036-1429,1095-7170},
   MRCLASS = {60J60 (35P15 62M05 65D99)},
  MRNUMBER = {4273695},
       DOI = {10.1137/20M1344093},
       URL = {https://doi.org/10.1137/20M1344093},
}

@article {MR3636868,
    AUTHOR = {Singer, Amit and Wu, Hau-Tieng},
     TITLE = {Spectral convergence of the connection {L}aplacian from random
              samples},
   JOURNAL = {Inf. Inference},
  FJOURNAL = {Information and Inference. A Journal of the IMA},
    VOLUME = {6},
      YEAR = {2017},
    NUMBER = {1},
     PAGES = {58--123},
      ISSN = {2049-8764,2049-8772},
   MRCLASS = {94A20},
  MRNUMBER = {3636868},
MRREVIEWER = {Luis\ G\'omez},
       DOI = {10.1093/imaiai/iaw016},
       URL = {https://doi.org/10.1093/imaiai/iaw016},
}

@incollection {MR2504294,
    AUTHOR = {Belkin, Mikhail and Sun, Jian and Wang, Yusu},
     TITLE = {Discrete {L}aplace operator on meshed surfaces [extended
              abstract]},
 BOOKTITLE = {Computational geometry ({SCG}'08)},
     PAGES = {278--287},
 PUBLISHER = {ACM, New York},
      YEAR = {2008},
      ISBN = {978-1-60558-071-5},
   MRCLASS = {65D17},
  MRNUMBER = {2504294},
       DOI = {10.1145/1377676.1377725},
       URL = {https://doi.org/10.1145/1377676.1377725},
}

@article{croke,
author = {Croke, Christopher},
year = {1980},
month = {01},
pages = {},
title = {Some isoperimetric inequalities and eigenvalue estimates},
volume = {13},
journal = {Annales Scientifiques de l’École Normale Supérieure. Quatrième Série},
doi = {10.24033/asens.1390}
}

@article {cheng,
    AUTHOR = {Cheng, Shiu Yuen},
     TITLE = {Eigenvalue comparison theorems and its geometric applications},
   JOURNAL = {Math. Z.},
  FJOURNAL = {Mathematische Zeitschrift},
    VOLUME = {143},
      YEAR = {1975},
    NUMBER = {3},
     PAGES = {289--297},
      ISSN = {0025-5874,1432-1823},
   MRCLASS = {58G99 (53C20 58E15)},
  MRNUMBER = {378001},
MRREVIEWER = {J.\ R.\ Vanstone},
       DOI = {10.1007/BF01214381},
       URL = {https://doi.org/10.1007/BF01214381},
}

@article {anderson,
    AUTHOR = {Anderson, Michael T.},
     TITLE = {Convergence and rigidity of manifolds under {R}icci curvature
              bounds},
   JOURNAL = {Invent. Math.},
  FJOURNAL = {Inventiones Mathematicae},
    VOLUME = {102},
      YEAR = {1990},
    NUMBER = {2},
     PAGES = {429--445},
      ISSN = {0020-9910,1432-1297},
   MRCLASS = {53C23 (53C21 58D27)},
  MRNUMBER = {1074481},
MRREVIEWER = {Gudlaugur\ Thorbergsson},
       DOI = {10.1007/BF01233434},
       URL = {https://doi.org/10.1007/BF01233434},
}

@article {MR407872,
    AUTHOR = {Dodziuk, Jozef},
     TITLE = {Finite-difference approach to the {H}odge theory of harmonic
              forms},
   JOURNAL = {Amer. J. Math.},
  FJOURNAL = {American Journal of Mathematics},
    VOLUME = {98},
      YEAR = {1976},
    NUMBER = {1},
     PAGES = {79--104},
      ISSN = {0002-9327,1080-6377},
   MRCLASS = {58A10 (58G99)},
  MRNUMBER = {407872},
MRREVIEWER = {D.\ B.\ Fuchs},
       DOI = {10.2307/2373615},
       URL = {https://doi.org/10.2307/2373615},
}

@incollection {MR448253berger,
    AUTHOR = {Berger, M.},
     TITLE = {Some relations between volume, injectivity radius, and
              convexity radius in {R}iemannian manifolds},
 BOOKTITLE = {Differential geometry and relativity},
    SERIES = {dath. Phys. Appl. Math.},
    VOLUME = {Vol. 3},
     PAGES = {33--42},
 PUBLISHER = {Reidel, Dordrecht-Boston, Mass.},
      YEAR = {1976},
   MRCLASS = {53C20},
  MRNUMBER = {448253},
MRREVIEWER = {Paul\ E.\ Ehrlich},
}

@article{aubry2013approximation,
  title={Approximation of the spectrum of a manifold by discretization},
  author={Aubry, Erwann},
  journal={arXiv preprint arXiv:1301.3663},
  year={2013}
}

@inproceedings{NIPS2006_5848ad95,
 author = {Belkin, Mikhail and Niyogi, Partha},
 booktitle = {Advances in Neural Information Processing Systems},
 editor = {B. Sch\"{o}lkopf and J. Platt and T. Hoffman},
 pages = {},
 publisher = {MIT Press},
 title = {Convergence of Laplacian Eigenmaps},
 url = {https://proceedings.neurips.cc/paper_files/paper/2006/file/5848ad959570f87753a60ce8be1567f3-Paper.pdf},
 volume = {19},
 year = {2006}
}

@Article{e21010043,
AUTHOR = {le Brigant, Alice and Puechmorel, Stéphane},
TITLE = {Approximation of Densities on Riemannian Manifolds},
JOURNAL = {Entropy},
VOLUME = {21},
YEAR = {2019},
NUMBER = {1},
ARTICLE-NUMBER = {43},
URL = {https://www.mdpi.com/1099-4300/21/1/43},
PubMedID = {33266759},
ISSN = {1099-4300},
DOI = {10.3390/e21010043}
}

@article {MR4384039,
    AUTHOR = {Calder, Jeff and Garc\'ia Trillos, Nicol\'as and Lewicka,
              Marta},
     TITLE = {Lipschitz regularity of graph {L}aplacians on random data
              clouds},
   JOURNAL = {SIAM J. Math. Anal.},
  FJOURNAL = {SIAM Journal on Mathematical Analysis},
    VOLUME = {54},
      YEAR = {2022},
    NUMBER = {1},
     PAGES = {1169--1222},
      ISSN = {0036-1410,1095-7154},
   MRCLASS = {35J15 (35R02 65N06 68T05)},
  MRNUMBER = {4384039},
       DOI = {10.1137/20M1356610},
       URL = {https://doi.org/10.1137/20M1356610},
}

@article {MR4279237,
    AUTHOR = {Dunson, David B. and Wu, Hau-Tieng and Wu, Nan},
     TITLE = {Spectral convergence of graph {L}aplacian and heat kernel
              reconstruction in {$L^\infty$} from random samples},
   JOURNAL = {Appl. Comput. Harmon. Anal.},
  FJOURNAL = {Applied and Computational Harmonic Analysis. Time-Frequency
              and Time-Scale Analysis, Wavelets, Numerical Algorithms, and
              Applications},
    VOLUME = {55},
      YEAR = {2021},
     PAGES = {282--336},
      ISSN = {1063-5203,1096-603X},
   MRCLASS = {62G08 (47A10)},
  MRNUMBER = {4279237},
       DOI = {10.1016/j.acha.2021.06.002},
       URL = {https://doi.org/10.1016/j.acha.2021.06.002},
}

\end{document}